\documentclass[final]{siamart0216}

\usepackage{pdfpages, amsfonts, amsmath, amssymb, color, graphicx, courier, listings, enumerate, bbold,pgf,tikz}
\usepackage{microtype, url, units, fancyhdr, fancybox, booktabs, thmtools, subfigure, caption, mathrsfs, breqn,dsfont,accents}
\usepackage{verbatim,epstopdf,algorithmic}
\newtheorem{assumption}[theorem]{Assumption}

\newcommand{\x}[1]{\textrm{#1}}

\newcommand{\su}[1]{\subsection{#1}}
\newcommand{\bv}{\begin{verbatim}}
\newcommand{\ev}{\end{verbatim}}
\newcommand{\im}{\textrm{im}}
\newcommand{\proj}{\textrm{proj}}

\newcommand{\rank}{\textrm{rank}}
\newcommand{\bp}{\begin{pmatrix}}
\newcommand{\ep}{\end{pmatrix}}
\newcommand{\bi}{\begin{itemize}}
\newcommand{\ei}{\end{itemize}}
\newcommand{\bc}{\begin{cases}}
\newcommand{\ec}{\end{cases}}
\newcommand{\ba}{\[\begin{aligned}}
\newcommand{\ea}{\end{aligned}\]}
\newcommand{\be}{\begin{enumerate}}
\newcommand{\ee}{\end{enumerate}}
\newcommand{\eqn}[2]{\begin{equation}\label{#2} #1\end{equation}}
\newcommand{\eq}[1]{\begin{align*}#1\end{align*}}
\newcommand{\eqq}[2]{\begin{align}\label{#2} \begin{split} #1 \end{split} \end{align}}
\newcommand{\sbq}{\subseteq}
\newcommand{\spq}{\supseteq}
\newcommand{\half}{\frac{1}{2}}
\newcommand{\q}{\quad}
\newcommand{\qq}{\qquad}
\newcommand{\R}{\mathds{R}}
\newcommand{\C}{\mathds{C}}

\newcommand{\N}{\mathds{N}}

\newcommand{\one}{\mathds{1}}
\newcommand{\I}{\mathds{I}}
\newcommand{\V}{\vspace{3mm}}

\newcommand{\bs}{\backslash}

\newcommand{\issa}{\\ = &}
\newcommand{\isen}{=&}

\newcommand{\ub}[2]{\underbrace{#1}_{#2}}
\newcommand{\haak}[1]{\left( #1 \right)}
\newcommand{\haakk}[1]{\left[ #1 \right]}

\newcommand{\acco}[1]{\left\{ #1 \right\}}
\newcommand{\lar}{\left( \begin{array}}
\newcommand{\rar}{\end{array} \right)}
\newcommand{\mc}[1]{\mathcal{#1}}

\newcommand{\md}[1]{\mathds{#1}}

\newcommand{\fone}[1]{\frac{1}{#1}}

\newlength{\dhatheight}
\newcommand{\doublehat}[1]{%
    \settoheight{\dhatheight}{\ensuremath{\hat{#1}}}%
    \addtolength{\dhatheight}{-0.35ex}%
    \hat{\vphantom{\rule{1pt}{\dhatheight}}%
    \smash{\hat{#1}}}}

\newcommand{\TheTitle}{TOPOLOGICAL AND GRAPH-COLORING CONDITIONS ON THE PARAMETER-INDEPENDENT STABILITY OF SECOND-ORDER NETWORKED SYSTEMS}
\newcommand{\TheAuthors}{F. J. Koerts, M. B{\"u}rger, A. J. van der Schaft and C. De Persis}

\newcommand{\TheTitleShort}{Networked Systems Containing Damped and Undamped Nodes}

\headers{\TheTitleShort}{\TheAuthors}

\title{{\TheTitle}\thanks{This work contains substantial overlap with a previous paper by the same authors \cite{acc}.}}

\author{
  Filip J. Koerts \thanks{Johann Bernoulli Institute for Mathematics and Computer Science, University of Groningen, The Netherlands
    (\email{f.j.koerts@rug.nl}).}
  \and
  Mathias B{\"u}rger  \thanks{Bosch Center for Artificial Intelligence, Robert Bosch GmbH (\email{mathias.buerger@de.bosch.com}).}
  \and
  Arjan J. van der Schaft \thanks{Johann Bernoulli Institute for Mathematics and Computer Science, University of Groningen, The Netherlands (\email{a.j.van.der.schaft@rug.nl}).}
    \and
    Claudio De Persis \thanks{Engineering and Technology Institute, University of Groningen, The Netherlands (\email{c.de.persis@rug.nl}).}
}

\usepackage{amsopn}

\ifpdf
\hypersetup{
  pdftitle={\TheTitle},
  pdfauthor={\TheAuthors}
}
\fi

\begin{document}

\maketitle

\begin{abstract}
In this paper, we study parameter-independent stability in qualitatively heterogeneous passive networked systems containing damped and undamped nodes. Given the graph topology and a set of damped nodes, we ask if output consensus is achieved for all system parameter values. For given parameter values, an eigenspace analysis is used to determine output consensus. The extension to parameter-independent stability is characterized by a coloring problem, named the richly balanced coloring (RBC) problem. The RBC problem asks if all nodes of the graph can be colored red, blue and black in such a way that (i) every damped node is black, (ii) every black node has blue neighbors if and only if it has red neighbors, and (iii) not all nodes in the graph are black. Such a colored graph is referred to as a richly balanced colored graph. Parameter-independent stability is guaranteed if there does not exist a richly balanced coloring. The RBC problem is shown to cover another well-known graph coloring scheme known as zero forcing sets. That is, if the damped nodes form a zero forcing set in the graph, then a richly balanced coloring does not exist and thus, parameter-independent stability is guaranteed. However, the full equivalence of zero forcing sets and parameter-independent stability holds only true for tree graphs. For more general graphs with few fundamental cycles an algorithm, named chord node coloring, is proposed that significantly outperforms a brute-force search for solving the NP-complete RBC problem. 
\end{abstract}




\section{Introduction}

This paper deals with the consensus or synchronization problem, where heterogeneous dynamical systems are coupled in such a way that they evolve asymptotically in an identical manner, see e.g. \cite{seyboth}, \cite{grip}. Synchronization is a fundamental stability-like property in numerous applications such as power systems, where frequencies of the power generators should be synchronized   \cite{bergenhill}, or platooning vehicles, where the velocities of the vehicles should be synchronized, see e.g. \cite{olfati}, \cite{traffic}.

We study in this paper a basic class of passive networks, namely linear mass-spring-damper networks with constant external forces. While this model is simple, it captures many of the relevant properties of networks of passive systems as studied in \cite{c1}, \cite{arcak}, \cite{port2}. We use output strict passivity of the nodes \cite{arcak}, \cite{port2}, and passivity of the couplings to ensure that consensus is achieved in the case that all nodes are damped.  However, in the presence of undamped nodes,  passivity is not output strictly anymore and consensus is not necessarily achieved. 

Mass-spring-dampers systems are systems with inertia, which involves second-order dynamics. Depending on the location of the damped nodes in the network, the system can show either convergence or unwanted oscillatory behavior for some initial conditions. Such oscillations do not appear in first-order models, and thus require a special analysis. In particular, the standard methods for the convergence analysis of passive networks are not applicable to our class of second-order systems.


Synchronization problems are particularly challenging if the individual systems are not identical but heterogeneous. There has been tremendous research on synchronization of heterogeneous systems (e.g., using dynamic coupling controllers) \cite{c2}, \cite{c1}, \cite{murguia}. In this paper, the heterogeneity is qualitative, as we consider mass-spring-type networks with typically many undamped and few damped nodes. This type of models might be applied to networks that contain a minority of nodes with damping constants being considerably higher than the damping constants of other nodes, in which case a natural approach would be to approximate the damping values below some threshold value by zero. 

Qualitatively heterogeneous systems can be also found in e.g. leader-follower systems where one is studying controllability properties. Here, the aim is to control the systems, where the input is applied to a subset of nodes called the leaders. The main result of \cite{acc} relates output consensus to observability of a Kron reduced system and using duality to controllability. When considering the whole network as one system, the controllability depends heavily on the location of the leaders in the network. In \cite{rahmani}, the controllability of leader-follower consensus networks has been connected to the symmetry of the graph with respect to the leaders. Similarly, the research direction of pinning control investigates the question, where to place a limited number of controllers in a network to achieve synchronization (see \cite{chen}, \cite{xiang} for a survey).

In the research field of strong structural controllability, one looks at controllability of a class of systems rather than a single system. This is often useful, as in many large-scale networked systems, the system parameter values are (partially) unknown, see e.g. \cite{ssc1}, \cite{ssc2}. In this paper, we assume that the system parameter values are completely unknown up to some feasibility constraints. In order to guarantee consensus, the system needs to satisfy stricter conditions than those described in \cite{acc} for the case of known values. {\color{blue} These stricter conditions are fundamentally different as they only involve the graph structure and the location of the damped nodes in the network.}

The contribution of the current paper includes a full characterization of these topological conditions, thereby giving an answer to the decision problem whether a system is parameter-independent globally asymptotically stable (PI-GAS). First, we show that this decision problem is equivalent to a graph coloring problem that we refer to as the richly balanced coloring (RBC) problem. Secondly, as this coloring problem turns out to be NP-complete, we discuss two graph coloring algorithms and show how they relate to the solution of the RBC problem.  

Inspired by previous work on zero forcing sets (see e.g. \cite{zfs1}, \cite{zfs2}, \cite{zfs3}), we study the zero forcing algorithm as an approximation algorithm for the RBC problem. In \cite{ssc2} and \cite{zfs4}, the zero forcing property was derived as a sufficient condition for strong structural target controllability for first-order systems. While these results do not apply to our second-order dynamics, we present a similar result, showing that the zero forcing property is in general also sufficient for parameter-independent stability. However, the property turns out not to be a necessary condition for PI-GAS, except for tree graphs. 

Motivated by this lack of necessity, we propose a second coloring algorithm, namely the novel chord node coloring (CNC) algorithm. This algorithm is proven to find the true solution to the RBC problem, and thus allows to identify all PI-GAS network topologies. It performs for certain networks, in particular for large networks with a limited number of cycles, significantly better than a brute-force search. To the best of our knowledge, the RBC problem and the CNC algorithm have not been presented in the literature before.

The paper is organized as follows. The dynamical network model and the graph formalism is introduced in \cref{sec2}. The system characteristics such as the network equilibrium and a shifted model are covered in \cref{sec3}. The convergence analysis is performed in \cref{sec4}, where first a Lyapunov analysis is presented, followed by a characterization of a certain invariant subspace, leading to a precise characterization of the convergence condition. This result is used in \cref{sec5} to show that the parameter-independent stability problem is equivalent to a graph coloring problem. Furthermore, the sufficient condition for parameter-independent stability based on zero forcing sets is  given, as well as the novel chord node coloring problem. The paper \cite{acc} contains a preliminary version of parts of Sections 2-4. 

{\bf Notation} For $v \in \R^n$ and $w \in \R^m$, by $col(v,w)$ we denote the vector $(v^T \; w^T)^T$. {\color{blue} The image and the kernel of a matrix $A$ is given by $\im(A)$ and $\ker(A)$, respectively.} The block diagonal matrix whose diagonal blocks are $A$ and $B$ is given by $diag(A,B)$. By $A \succ 0$ and $A \succcurlyeq 0$ we denote positive definiteness and semi-positive definiteness of $A$, respectively. Moreover, $\sigma(A)$ is the spectrum of $A$, i.e. the set of eigenvalues of $A$.  {\color{blue} For two subspaces $A, B \sbq \R^n$, the subspace sum $A \oplus B$ is the span of the union $A \cup B$.}The quotient space of the vector space $\R^n$ by a subspace $S$, i.e. the set of affine subspaces in $\R^n$ parallel to $S$, is denoted $\R^n / S$. {\color{blue}The Kronecker product of $A$ and $B$ is denoted by $A \otimes B$} and $A^\dagger$ denotes the Moore-Penrose pseudoinverse of $A$. $\one$ is the vector whose entries are all 1. The identity matrix of size $n$ is denoted by $I_n$. {\color{blue} For a graph $\mc G=(V,E)$, the (oriented) incidence matrix $B= B(\mc G) \in \R^{|V| \times |E|}$ shows the connection of the edges and vertices in such a way that every column contains exactly one 1 and one -1 in the rows corresponding to its endpoints, while all other entries are zero.} The observability matrix of the system $(C,A)$ is denoted by $Obs(C,A)$.

\section{Preliminaries}\label{sec2}

We consider a mass-spring-damper system defined on an undirected and connected graph $\mc G=(V,E)$ with $n$ nodes and $m$ edges, and incidence matrix $B$. On each node $i \in V$, a mass $\Sigma_i$ is placed which is modeled as
\eqq{ 
\dot p_i = -R_iy_i + u_i + v_i, \qquad y_i = M_i^{-1} p_i
}{nodes}
where $p_i \in \R^{r}$ and  $y_i \in \R^{r}$ are the momentum (state) and velocity (output) of the masses, respectively. Further, we have the damping matrix $R_i$, inertia matrix $M_i$, coupling force (coupling input) $u_i$ and a constant external force (external input) $v_i \in \R^r$. On each edge $k=(i,j) \in E$, a spring $\Gamma_k$ of dimension $r$ is placed with elongation (state) $q_k \in \R^{r}$, output force $f_k \in \R^{r}$, which is modeled as:
\eqq{ 
\dot q_k = {\color{blue}{\zeta_k}}, \qquad f_k \isen W_kq_k
}{edges}
Here, $W_k$ is the edge weight matrix of edge $k$ and {\color{blue} $\zeta_k$ is the input force}. Variables without subscript denote the corresponding stacked variables of the masses and springs. The coupling is established through 
\eqq{
u = -(B \otimes I_{r}) f, \qquad {\color{blue}\zeta} = (B^T \otimes I_{r}) y
}{coupling}
In the sequel, we will use the abbreviated notation $\mathds{B}:= B \otimes I_{r}$.  

\begin{assumption}\label{denass} The damping matrix, inertia matrix and edge weight matrix satisfy $R_i \succcurlyeq 0$, $M_i \succ 0$, $W_i \succ 0$. 
\end{assumption}

\begin{definition} \textit{A node $i \in V$ is said to be damped if $R_i \succ 0$. A node is undamped if $R_i =0$, while it is partially undamped if $R_i$ is nonzero and singular. }
\end{definition} 

For any partially undamped node, there exist directions for which it does not experience damping (namely, every direction in the kernel of $R_i$) and directions for which it does (any other direction).  

\begin{assumption} The set $V$ of nodes is partitioned into a set $V_d$ of damped nodes with cardinality $n_d \geq 1$ and a set $V_u$ of (partially) undamped nodes with cardinality $n_u \geq 0$. \footnote{The case with $n_u=0$ leads to trivial results, but is required for the analysis in Chapter 5.}
\end{assumption}

{\bf Remark }  
Mass-spring-damper systems are used here as a leading example. Other examples such as hydraulic systems can be modelled similarly.


\subsection{Closed-loop system}

Let $p=col(p_1, \dots, p_n)$, $q=col(q_1, \dots, q_m)$ be the stacked state vectors and similarly for the other variables. Taking \cref{nodes}, \cref{edges} and \cref{coupling} together, we obtain the closed-loop system, denoted by $\Sigma \times \Gamma$ and whose state and output is denoted by $z:=col(p,q) \in \R^{r(n+m)}$ and $y \in \R^{rn}$, respectively. Its state-space representation reads as $\dot z=Az+Gv$, $y=M^{-1}p$, where

\eqq{
\ub{\bp \dot p \\ \dot q \ep}{\dot z} \isen \ub{\bp -RM^{-1} & -\md BW \\ \md B^TM^{-1} & 0 \ep}{A} \ub{\bp p \\ q \ep}{z} + \ub{\bp I \\ 0 \ep}{G} v \\
y \isen \bp M^{-1} & 0 \ep \bp p \\ q \ep 
}{y}

The system parameters are:
\bi 
\item $M:=diag(M_1, \dots, M_n) \succ 0$, a block diagonal matrix containing inertia matrices of the individual nodes.  
\item $R:=diag(R_1, \dots, R_n) \succcurlyeq 0$, a block diagonal matrix with damping matrices of the individual nodes. 
\item $W:=diag(W_1, \dots W_m) \succ 0$, the block diagonal matrix with spring constants (edge weights). 
\item $v$, a constant external input 
\ei 
\medskip

{\bf Remark } The system $\Sigma \times \Gamma$ can be written in a port-Hamiltonian representation, where $H(p,q)=\half p^TMp + \half q^TWq$ is used as Hamiltonian function. This gives $\dot H=v^Ty-y^TRy \leq v^Ty$, which shows that $\Sigma \times \Gamma$ is passive, but not output strictly passive as $R$ is singular. Hence, this does not give us the wanted convergence results and we need to invoke LaSalle's Theorem (section 3).

\subsection{Second-order dynamics}

Since $\mc G$ is connected, $\x{rank}(B)=n-1$. Furthermore, $\ker(B^T)=\im(\one)$, where $\one$ is the stacked vector of all ones. {\color{blue} Also, $\ker(\mathds{B}^T)=\im(\I_r)$}, where $\I_r:= \one \otimes I_r$. In fact, $\md B$ represents a graph consisting of $r$ connected components that are copies of $\mc G$. 

A fundamental cycle matrix $C$ of $\mc G$ is a matrix of full column rank that satisfies $\ker(B)=\im(C)$, see e.g. \cite{graphmatrices}\footnote{The fundamental cycle matrix in \cite{graphmatrices} is the transposed of the fundamental cycle matrix used in this paper.}. The full column rank matrix $\mathds{C}:= C \otimes I_r$ satisfies $\ker(\mathds{B})=\im(\mathds{C})$. 

Note that since $\dot q \in \im(\mathds{B}^T)$, the projection of $q$ onto $\im(\mathds{B}^T)^\perp=\ker(\mathds B)=\im(\mathds C)$ can be written as $\mathds C r$ for some $r \in \R^{r(m-n+1)}$ and is such that $q(t) \in \im(\mathds{B}^T) + \mathds{C} r$ for all $t \geq 0$. By integrating the output $y$, we obtain potentials or positions $s(t):= \int_0^t y(\tau) d\tau+s_0$, where $s_0$ satisfies $q(0)=\mathds{B}^Ts_0+W^{-1}\mathds{C} r$. This decomposition is possible and unique. All terms in the equation $\dot p=-Ry -\mathds{B}Wq +v$ can be written in terms of (derivatives) of $s$ and the result is a second order equation: 
\eqq{
M \ddot s = -R \dot s - \mathds{B}W\mathds{B}^T s + v
}{sor}
We define $\md L:= \md B W \md B^T$ to be the total Laplacian matrix. Also, $\ker(\md L)=\ker(\md B^T)=\im(\I_r)$. The graph $\mc G_{\md L}$ associated with $\md L$ might be disconnected. This is the case if e.g. $r>1$ and all $W_i$'s are diagonal. Then $\mc G_{\md L}$ consists of $r$ connected components that are copies of $\mc G$. If the matrices of $W_i$ are full, then $\mc G_{\md L}$ is the strong product of $\mc G$ and the complete graph with $r$ nodes\footnote{For more on strong products of graphs, see \cite{sabidussi}. In Definition 1.1(2) of this paper, take $B=\{\mc G, \mc G_r\}$ and $A=\emptyset$, with $\mc G_r$ being the complete graph with $r$ nodes.}. Off-diagonal entries of $\md L$ can be positive, which occurs if and only if there are $W_k$'s with negative off-diagional entries. This does not affect the stability, since $W$ is positive-semidefinite (see section 4). 

\subsection{Decomposition of $\md B$ and $\md L$}

The partitioning $\{V_d , V_u\}$ of $V$ also induces a partitioning of the edges into the set $E^d$ of edges between damped nodes, the set $E^u$ of edges between undamped nodes and the set $E^i$ of interconnecting edges between a damped and an undamped node. We obtain $\mc G=(V_d \cup V_u, E^d \cup E^i \cup E^u)$ with partitioned total incidence matrix
\eq{
\md B= \bp \md B_d \\ \md B_u \ep =  \bp \md B_d^d & \md B^i_d & 0 \\ 0 & \md B^i_u & \md B^u_u \ep 
}
Let the edge weight matrix $W$ and the total Laplacian matrix $\md L$ be correspondingly decomposed. Now, decompose \cref{sor} into blocks associated with the damped nodes, with subscript $d$, and (partially) undamped nodes, with subscript $u$, as follows 
\eqq{
 \underbrace{\bp M_d & 0 \\ 0 & M_u \ep}_{M} \bp \ddot s_d \\ \ddot s_u \ep =& - \underbrace{\bp R_d & 0 \\ 0 & R_u \ep}_{R} \bp \dot s_d \\ \dot s_u \ep \\ & \hspace{-10mm} - \underbrace{\bp \md L^d_d+\md L_d^i & \md L^i_i \\ (\md L^i_i)^T & \md L^u_u+\md L_u^i \ep}_{L} \bp s_d \\ s_u \ep + \bp v_d \\ v_u \ep
}{el}
$\md L^d_d=\md B_d^dW^d(\md B_d^d)^T$ and $\md L^u_u=\md B_u^uW^u(\md B_u^u)^T$ are the Laplacian matrices corresponding to the subgraphs $\mc G_d:=(V_d,E^d)$ with incidence matrix $\md B_d^d$ and $\mc G_u:=(V_u,E^u)$ with incidence matrix $\md B_u^u$, respectively. The Laplacian matrix 
$$ \md L^i := \bp \md L_d^i & \md L_i^i \\ (\md L_i^i)^T & \md L_u^i \ep$$ 
corresponds to the subgraph $\mc G_i:=(V,E^i)$. $\md L_i^i=\md B_d^iW^i(\md B_u^i)^T$ contains the edge weight matrices of the interconnecting edges. Finally, $\md L_d^i=\md B_d^iW^i(\md B_d^i)^T$ and $\md L_u^i=\md B_u^iW^i(\md B_u^i)^T$ are positive semi-definite block diagonal matrices since by definition of $\mc G_i$, there are no edges between two damped or two undamped nodes in $\mc G_i$.


\subsection{Stability analysis}
Let $\sigma(A)$ denote the spectrum of the matrix $A$. In the stability analysis, we refer to Barbalat's and Hautus Lemma, which are stated below: 
\begin{lemma}[Barbalat's Lemma  \cite{vukic}]\label{barbalat} Let $f(t)$ be a function of time only. If $f(t)$ has a finite limit as $t \to \infty$ and moreover, $\dot f$ is uniformly continuous or $\ddot f$ is bounded, then $\dot f(t) \to 0$ as $t \to \infty$. 
\end{lemma}

\begin{lemma}[Hautus Lemma \cite{hautus}]\label{hautus} Consider matrices $A \in \R^{n \times n}$ and $B \in \R^{m \times n}$. Then the following is equivalent: 
\bi
\item The pair $(A,B)$ is controllable
\item For all $\lambda \in \C$, it holds that $\rank \bp \lambda I- A & B \ep =n$. 
\item For all $\lambda \in \sigma(A)$, it holds that $\rank \bp \lambda I- A & B \ep =n$. 
\ei
\end{lemma}

\subsection{Problem formulation}\label{secprobleem}

Suppose that we know the topology of the graph including the set $V_d$ of damped nodes, but we do not know the system parameter values, except that they are feasible, i.e. \cref{denass} is met. The graph topology and the set of damped nodes define a damping graph represented as $\mc G=(V, V_d, E)$, where $V_d \sbq V$. Output consensus is achieved if the output variable $y$ converges to a point in the synchronization manifold $\im\{\I_r\}$ as $t \to \infty$. In this sense, it is of interest to know whether a given damping graph gives rise to asymptotic output consensus for all feasible system parameter values and initial conditions. Thus, in this paper we address the following problem: 

\medskip

{\bf Problem} \textit{(Parameter-independent stability problem)} Under which conditions on the damping graph $\mc G=(V,V_d,E)$ does every output trajectory $y(t)$ of $\Sigma \times \Gamma$, i.e. system \cref{y}, converge to a point in $\im\{\I_r\}$ as $t \to \infty$ for all feasible system parameter values?


\medskip

In section 3 and 4 we give a graph theoretic answer to the question whether output consensus is guaranteed for a \textit{given} set of system parameter values, {\color{blue} which has been elaborated in \cite{acc}. This result is exploited in section 5 as a starting point to solve the parameter-independent stability problem. }



\section{System characteristics}\label{sec3}

In this section, we determine the equilibria and perform a shift so that the equilibrium is located at the origin. Note first that the affine subspaces of  $\mathds{R}^{rm} \backslash \im(\mathds{B}^T)$ are invariant under the dynamics of the controller state $q(t)$. We have the following characterization of $\mathds{B}$ and $\mathds{C}$: 

\begin{lemma}\label{bandc} $\im(\mathds{B}^T) \oplus  \im(W^{-1}\mathds{C}) = \mathds{R}^{rm}$ and $\im(\mathds{B}^T) \cap  \im(W^{-1}\mathds{C}) = \{0\}$. 
\end{lemma}

\begin{proof} The first statement follows easily from $\im(\mathds{B}^T)^\perp = \ker(\mathds{B}) = \im(\mathds{C})$ and $W^{-1}$ is positive definite. The second statement holds since $\im(\mathds{B}^T) \cap \im(W^{-1} \mathds{C}) = \im(\mathds{B^T}) \cap \ker(\mathds{B}W)=\{0\}$. 
\qed \end{proof}

As a result, the set of solutions $z(t)=col(p(t),q(t))$ of $\Sigma \times \Gamma$ where $q(t)$ is in one of the affine subspaces of $\mathds{R}^{rm} \backslash \im(\mathds{B}^T)$ is a shifted copy of those solutions of $\Sigma \times \Gamma$ where $q(t) \in \im(\md B^T)$. Consequently, without loss of generality we can assume that $q \in \im(\md B^T)$.

\begin{corollary}\label{copy} For every initial condition $z(0)=col(p(0), q(0))\in \mathds{R}^{r(n+m)}$, there exists a unique vector $\gamma \in \R^{r(m-n+1)}$ such that $q^*(0) := q(0) - W^{-1}\mathds{C} \gamma \in \im(\mathds{B}^T)$. Furthermore, with shifted initial conditions $z^*(0)=col(p(0), q^*(0))$, the trajectory difference $z^*(t) - z(t)$ is constant for all $t \geq 0$. 
\end{corollary}

\begin{proof} Existence and uniqueness of $\gamma$ follow from \cref{bandc}


Given that at some time $t \geq 0$, $p^*(t)=p(t)$ and $q^*(t)-q(t) \in \im(W^{-1}\mathds{C})$, we have $y^*(t)=y(t)$ and $\mathds{B}W(q^*(t)-q(t))=0$, hence $\dot p^*(t) -\dot p(t) = 0$. Also, $\dot q^*(t) - \dot q(t)= 0$.
\qed\end{proof}

If we restrict $q$ to be in the invariant space of $\im(\md B^T)$, the system $\Sigma \times \Gamma$ has a unique equilibrium:

\begin{proposition}
The system $\Sigma \times \Gamma$ restricted to $q \in \im(\mathds{B}^T)$ has a unique equilibrium $\bar z = col(\bar p, \bar q)$ satisfying $\bar p=M \I_r \beta$ and \\ 
$\{ \bar q \} = \haakk{W^{-1}\md B^\dagger(-R \I_r \beta +v) + \im(W^{-1}\md C)} \cap \im(\md B^T)$.\footnote{Here, $\dagger$ denotes the Moore-Penrose pseudoinverse.} In these expressions, $\beta$ is given by
\eqq{ \beta= ( \mathds{I}_r^T R \mathds{I}_r)^{-1} \mathds{I}_r^Tv
}{beta}
\end{proposition}



\begin{proof} From $\dot{\bar q}=0$, we obtain $\mathds{B}^TM^{-1}\bar p=0$. Since for connected graphs, $\ker(\mathds{B}^T)=\im \{ \I_r \}$, it follows that $\bar p \in \im \{ M \I_r \}$. Write $\bar p= M \I_r \beta$ for some $\beta \in \mathbb{R}^r$. The value of $\beta$ can be obtained by setting $\dot{\bar p}=0$, which gives $-R \I_r \beta + v \in \im(\md B)$ and consequently $-\I_r^TR \I_r \beta+ \I_r^T v \in \im(\I_r^T\md B)=\{0\}$. Since $R_i \succ 0$ for at least one $i \in V$, it follows that $\I_r^TR \I_r \succ 0$ and thus $\beta$ can be given uniquely as in \cref{beta}. Substituting this result in the dynamics of $p$, we obtain $\md BW \bar q = -R \I_r \beta +v$, which gives $\bar q \in W^{-1}\md B^\dagger (-R \I_r \beta + v) + \im(W^{-1} \md C)$, with $\im(W^{-1} \md C) = \ker( \md B W)$. By assumption and by the dynamics of $q$, $\bar q \in \im(\md B^T)$. The intersection of both sets is a singleton by \cref{bandc}. 
\qed\end{proof}

{\bf Remark } The unique equilibrium point for $q(0) \in \im(\md B^T)$ corresponds to a state of output consensus since  $\bar v = M^{-1}\bar p \in \im\{ \I_r \}$. Also for $q(0) \notin \im(\md B^T)$, it is shown readily that $\bar v = \I_r \beta$ with $\beta$ being given by \cref{beta}.

\subsection{Shifted model}

Now, we introduce shifted state variables so that the equilibrium coincides with the origin. The main benefit of doing this is that it allows to use common techniques to show output consensus in section 4.2. Besides that, we get rid of the constant input $v$ in the dynamics. Define $\tilde p(t)= p(t) - \bar p$, $\tilde q(t) =q(t) - \bar q$. Stack these together in the state vector $\tilde z=col(\tilde p, \; \tilde q)$ and define the output $\tilde y=M^{-1}p$, then we obtain the linear time-invariant (LTI) closed-loop system
\eqq{
\ub{\bp \dot{\tilde p} \\ \dot{\tilde q} \ep}{\dot{\tilde z}} \isen \ub{\bp -RM^{-1} & -\md BW \\ \md B^TM^{-1} & 0 \ep}{A} \ub{\bp \tilde p \\ \tilde q \ep}{\tilde z}  \\
\tilde y \isen \bp M^{-1} & 0 \ep \bp \tilde p \\ \tilde q \ep 
}{y2}

Since $\bar q \in \im(\md B^T)$, it follows that $q(t) \in \im(\md B^T)$ if and only if $\tilde q(t) \in \im(\md B^T)$. By assumption, $\Sigma \times \Gamma$ is only defined on the invariant subspace $$\Omega = \{ col(\tilde p, \; \tilde q) \in \R^{r(n+m)} \mid  \tilde q \in \im(\md B^T) \}$$
System \cref{y2} defined on $\Omega$ has a unique equilibrium point at $\bar{\tilde z}=0$. Similarly to the procedure in section 2.2, we can introduce variables $\tilde s(t) = col(\tilde s_d(t), \tilde s_u(t)) := \int \tilde y(\tau) d\tau + \tilde s_0$, where $\tilde s_0$ is such that $\tilde q(0)=\md B^T \tilde s_0$, to obtain the second-order equation $$M\ddot{\tilde s}=-R\dot{\tilde s}- \md L{\tilde s}$$ 
Noting that $\tilde q(t)=\md B^T \tilde s(t)$ and $\tilde p(t)=M \dot{\tilde s}(t)$, \cref{y2} can be written equivalently as an LTI system with states $\tilde s$ and $\tilde y$. In the next section we find \cref{qook} that connects global asymptotic stability of \cref{y2} defined on $\Omega$ with output consensus of \cref{y}.

\section{Steady-state behavior}\label{sec4}

In this section, we determine the long-run behavior of \cref{y2} defined on $\Omega$ by performing a common Lyapunov analysis. This allows us to derive the set of points to which all solutions converge. To find necessary conditions for the steady-state behavior, we use as Lyapunov function the Hamiltonian function that has a minimum at the equilibrium point $(\bar{\tilde p}, \bar{\tilde q})=(0,0)$:
$$U(\tilde p,\tilde q)=\frac{1}{2} \tilde p^TM^{-1}\tilde p+ \frac{1}{2} \tilde q^TW\tilde q $$

The time derivative of $U(\tilde p, \tilde q)$ now reads as
\begin{align*}
\dot U(\tilde p,\tilde q) &= \dot{\tilde p}^TM^{-1}\tilde p+ \dot{\tilde q}^TW\tilde q \\
&= (-RM^{-1} \tilde p- \md BW\tilde q)^T   M^{-1}{\tilde p} +  {\tilde p}^T M^{-1} \md BW\tilde q  \\
&= -\bp  {\tilde p}_d^T & {\tilde p}_u^T \ep M^{-1} \bp R_d & 0\\ 0 & R_u \ep M^{-1} \bp  {\tilde p}_d \\ {\tilde p}_u \ep \\
&=- {\tilde p}_d^T M_d^{-1}R_dM_d^{-1}  {\tilde p}_d - \tilde p_u^T M_u^{-1} R_u M_u^{-1} \tilde p_u 
\end{align*}

From the fact that $M$ and $W$ are positive definite, $U$ is a postive-definite function for $(\tilde p,\tilde q) \neq (0,0)$, while $\dot U$ is negative semi-definite, $U$ is a suitable Lyapunov function. Since $U$ is a radially unbounded function, it follows that the system $\dot{\tilde z}=A \tilde z$ as defined in \cref{y2}, is stable. 



\begin{lemma}\label{qook}
Every trajectory $y(t)$ of $\Sigma \times \Gamma$, i.e. system \cref{y}, converges to a point in the set $\im\{ \I_r \}$ if and only if every trajectory $\tilde z(t)=col( \tilde p(t),  \tilde q(t))$ of \cref{y2} defined on $\Omega$ converges to the origin. 
\end{lemma}

\begin{proof}  In this proof, every convergence statement holds exclusively for $t \to \infty$.
$(\Rightarrow)$ Suppose that every output trajectory $y(t)$ of $\Sigma \times \Gamma$ converges to $\im\{ \I_r \}$, then this holds in particular for those trajectories generated with $q(0) \in \im(\md B^T)$. So $p(t) \to p^*$, where $p^* \in \im\{ M \I_r \}$. Since $\Sigma \times \Gamma$ is stable, $\ddot{p}(t)$ is bounded, hence $\dot p(t)$ is uniformly continuous and we can apply Barbalat's lemma to conclude that $\dot p(t) \to 0$. That gives $-RM^{-1}p(t)-\md BW  q(t)+v \to 0$. So $\md BWq(t)$ converges too and consequently, $q(t) \to q^* + \ker(\md BW)$ for some $q^* \in \R^m$. Since $q(t) \in \im(\md B^T)$, we have that $q(t) \to \im(\md B^T) \cap \haakk{q^* + \ker(\md BW)}$, which is a singleton, so $q(t)$ converges too and consequently, the whole state $z(t)$ converges. Uniqueness of the equilibrium implies that $z(t) \to \bar z$ and hence $\tilde z(t) \to 0$.
$(\Leftarrow)$ For every trajectory $\tilde z(t) \to 0$ we have $z(t) \to \bar z$, yielding $y(t) \to \bar y=M^{-1}\bar p =\I_r \beta$, with $\beta$ as in \cref{beta}.
\qed\end{proof}

\cref{qook} shows that asymptotic output consensus of the system $\Sigma \times \Gamma$ is equivalent to global asymptotic stability (GAS) of the system $\cref{y2}$ defined on $\Omega$ and we will interchangeably use both terms.

Now we use LaSalle's invariance principle: as $t$ goes to infinity, the trajectory converges to the largest invariant set in the set of states where $\dot U=0$. 

\begin{lemma}\label{pede} 
Let $\mc S^{LS}$ be the set of initial conditions $\tilde z(0)=col(\tilde p(0),\tilde q(0)) \in \Omega$ of the system \cref{y2} such that for all $t \geq 0$, $\tilde p_d(t)=0$ and $\tilde p_u(t) \in \ker(R_u M_u^{-1})$. The largest invariant set contained in the subset of $\Omega$ where $\dot U=0$ equals $\mc S^{LS}$. \end{lemma}

\begin{proof} The set of points in $\Omega$ where $\dot U=0$, is equal to the set of points in $\Omega$ where ${\tilde p}_d =0$ and $\tilde p_u \in \ker(R_uM_u^{-1})$. This follows by positive definiteness of $M_d^{-1}R_dM_d^{-1}$ and by $\ker(M_u^{-1}R_uM_u^{-1})=\ker(R_uM_u^{-1})$. The largest invariant subset of the set of states where $\dot U=0$ is then precisely $\mc S^{LS}$. 
 \qed\end{proof}

\subsection{Behavior of the undamped nodes at steady state}

In this subsection, we give a precise characterization of $\mc S^{LS}$ as defined in \cref{pede} and work towards an LTI system that is observable if and only if output consensus of $\Sigma \times \Gamma$ is achieved. Starting with \cref{y2}, we derive the following explicit characterization of $\mc S^{LS}$:

\begin{lemma}\label{bestbelangrijk} $\mc S^{LS} = \hat Q \ker(Obs(\hat C, \hat A))$, with $\mc S^{LS}$ as defined in \cref{pede} and
$$\hat Q= \bp 0_{n_dr \times n_ur} & 0_{n_dr \times n_ur} \\ M_u & 0_{n_ur \times n_ur} \\ 0_{mr \times n_ur} & \md B_u^T- \md B_d^T(\md L_d^d+\md L_d^i)^{-1}\md L_i^i \ep ,$$ 
\eq{
\hat C =  \bp  \md L_i^i & 0_{n_dr \times n_ur} \\ R_u & 0_{n_ur \times n_ur} \ep, \qquad
\hat A = \bp 0_{n_ur \times n_ur} &  -M_u^{-1} \tilde{\md L}_u \\ I_{n_ur} & 0_{n_ur \times n_ur} \ep, 
}
\eqq{
\tilde{\md L}_u = (\md L_u^u+ \md L_u^i) - (\md L_i^i)^T (\md L_d^d+\md L_d^i)^{-1} \md L_i^i
}{tlu}
Here, the block rows of $\hat Q$ are decomposed according to the decomposition of $\tilde z(t) = col( \tilde p_d, \tilde p_u, \tilde q )$ and the block columns of $\hat Q$ according to the decomposition of $\hat A$. 
\end{lemma}

\begin{proof}
$(\sbq)$ Take $\tilde z(0) \in \mc S^{LS}$ and consider the trajectory $\tilde z(t) = e^{At} \tilde z(0) \subset \Omega$, which is a solution to \cref{y2}. Decompose $\tilde z$ as $col(\tilde p_d, \tilde p_u, \tilde q)$, then we have for all $t \geq 0$: $\tilde p_d(t)=0$, $\tilde p_u(t) \in \ker(R_uM_u^{-1})$, $\tilde q(t) \in \im(\md B^T)$ and
 $$\bp \dot{ \tilde p}_d \\ \dot{\tilde p}_u \\ \dot{\tilde q} \ep = \bp -R_dM_d^{-1} & 0 & -\md B_dW \\ 0 & -R_uM_u^{-1} & -\md B_uW \\ \md B_d^T M_d^{-1} & \md B_u^T M_u^{-1} & 0 \ep \bp { \tilde p}_d \\ {\tilde p}_u \\ {\tilde q} \ep$$
From $\tilde p_d \equiv 0$, it also follows that $\dot{\tilde p}_d \equiv 0$. Also, $-R_uM_u^{-1}\tilde p_u \equiv 0$. Substituting these results in the dynamics, we obtain that for all $t \geq 0$, $-\md B_dW \tilde q(t) = 0$ and 
\eqq{
\bp \dot{\tilde p}_u \\ \dot{\tilde q} \ep = \bp 0 & - \md B_uW \\ \md B_u^TM_u^{-1} & 0 \ep \bp \tilde p_u \\ \tilde q \ep
}{tored}
Define the function $\tilde y_u(t):=M_u^{-1}\tilde p_u(t)$ and the auxiliary function $\tilde y_d(t) := M_d^{-1} \tilde p_d(t) \equiv 0$. It follows directly that
\eqq{
R_u \tilde y_u(t) \equiv 0
}{ru}
Since $\tilde q \in \im(\md B^T)$, there exist initial positions $\tilde s_d(0)$ and $\tilde s_u(0)$ satisfying 
\eqq{
\tilde q(0) = \md B_d^T \tilde s_d(0) + \md B_u^T \tilde s_u(0)
}{qq}
Now, define the functions $\tilde s_u(t) := \int_0^t \tilde y_u(\tau) d \tau + \tilde s_u(0)$ and $\tilde s_d(t) := \int_0^t \tilde y_d(\tau) d \tau + \tilde s_d(0) \equiv \tilde s_d(0)$. From $\dot{\tilde q}(t) = \md B_u^T \tilde y_u(t)$, we have $\tilde q(t) = \int_0^t \md B_u^T \tilde y_u(\tau) d \tau + \tilde q(0)$, which is easily shown to satisfy $\tilde q(t)=\md B_d^T \tilde s_d(t) + \md B_u^T \tilde s_u(t)$ for all $t \geq 0$. 

From $\md B_dW \tilde q(t) \equiv 0$, it follows that $-\md B_dW\md B_d^T \tilde s_d(t) = \md B_dW\md B_u^T \tilde s_u(t)$, which, after exploiting the decomposition of $\md L$ as in \cref{el}, results in 
\eqq{
\tilde s_d(t) \equiv -(\md L_d^d+\md L_d^i)^{-1} \md L_i^i \tilde s_u(t)
}{sdd} Also, from $-\md B_dW \tilde q(t) \equiv 0$, it follows that $-\md B_dW \dot{\tilde q}(t) = -\md B_dW \md B_u^T M_u^{-1} \tilde p_u(t)\equiv 0$, or equivalenlty, 
\eqq{
-\md L_i^i \tilde{ y}_u(t) \equiv 0
}{nul}
The first row of the dynamics \cref{tored} can now be rewritten in the new variables: 
$$ M_u \dot{\tilde y}_u(t) = -\md B_uW(\md B_d^T \tilde s_d(t) + \md B_u^T \tilde s_u(t))$$
By construction, $\dot{\tilde s}_u(t) = \tilde y_u(t)$. Using \cref{sdd} and \cref{tlu}, the dynamics of $\tilde s_u$ and $\tilde y_u$ satisfy 
\eqq{
\bp \dot{\tilde y}_u \\ \dot{\tilde s}_u \ep = \bp 0 & -M_u^{-1} \tilde{\md L}_u \\ I_{n_ur} & 0 \ep \bp \tilde y_u \\ \tilde s_u \ep
}{dyna}
From \cref{ru}, \cref{nul} and \cref{dyna}, it follows immediately that $col(\tilde y_u(t),  \tilde s_u(t) ) \in \ker(Obs(\hat C, \hat A))$ for all $t \geq 0$, which holds in particular for $t=0$. Finally, by combining \cref{qq} and \cref{sdd} for $t=0$, we see that $\tilde q(0)=(\md B_u^T-\md B_d^T(\md L_d^d+\md L_d^i)^{-1} \md L_i^i) \tilde s_u(0)$. We conclude that
\eqq{
\bp \tilde p_d(0) \\ \tilde p_u(0) \\ \tilde q(0) \ep  \in \hat Q \ker(Obs(\hat C, \hat A))
}{itis}

$(\spq)$ Consider any vector $col(\tilde y_u(0), \; \tilde s_u(0) ) \in \ker(Obs(\hat C, \hat A))$ and the trajectory $col(\tilde y_u(t), \; \tilde s_u(t) ) = e^{\hat A t} col(\tilde y_u(0), \tilde s_u(0) )$. Note that $\hat Q \: col(\tilde y_u(0), \; \tilde s_u(0) ) \in \mc S^{LS}$ if and only if the trajectory
\eqq{
\tilde z(t)=\bp \tilde p_d(t) \\ \tilde p_u(t) \\ \tilde q(t) \ep  := \ub{ \bp 0 & 0 \\ M_u & 0 \\ 0 & \md B_u^T-\md B_d^T(\md L_d^d+\md L_d^i)^{-1}\md L_i^i \ep}{\hat{Q}}  \bp \tilde y_u(t) \\  \tilde s_u(t) \ep
}{newtraj}
is contained in ${\mc S}^{LS}$ and satisfies the dynamics of system \cref{y2}. Firstly, $\tilde p_d(t)=0$ and $\tilde q(t) \in \im(\md B^T)$ for all $t \geq 0$. Secondly, by unobservability we have $R_u \tilde y_u(t) = R_uM_u^{-1} \tilde p_u(t) \equiv 0$. It remains to show that $\dot{\tilde z}(t) = A \tilde z(t)$ for all $t \geq 0$.  Note that $\md B_dW \tilde q(t) =  \md L_i^i \tilde s_u(t)  - \md L_i^i \tilde s_u(t)=0$. Therefore, 
\eqq{
\dot{\tilde p}_d(t) = 0 = \ub{- R_dM_d^{-1} \tilde p_d(t)}{=0}  - \ub{\md B_dW \tilde q(t)}{=0}
}{x1}
From $\tilde q(t)=(\md B_u^T-\md B_d^T(\md L_d^d+\md L_d^i)^{-1}\md L_i^i)\tilde s_u(t)$, it follows that  $\md B_uW \tilde q(t)= \tilde{\md L}_u \tilde s_u(t)$. Also, $\dot{\tilde y}_u(t) = -  M_u^{-1} \tilde{\md L}_u \tilde s_u(t)$, hence: 
\eqq{
\dot{\tilde p}_u(t) = M_u \dot{\tilde y}_u(t) = \ub{-R_uM_u^{-1} \tilde p_u(t)}{=0}- \md B_uW \tilde q(t)
}{x2}
By unobservability, $\md L_i^i \tilde y_u(t) = \md L_i^i \dot{\tilde s}_u(t) \equiv 0$. As a consequence, 
$\dot{\tilde q}(t) = \md B_u^T \dot{\tilde s}_u(t) = \md B_u^T \tilde y_u(t) = \md B_u^T M_u^{-1} \tilde p_u(t)$ and therefore
\eqq{
\dot{\tilde q}(t) = \ub{\md B_d^TM_d^{-1}\tilde p_d(t)}{=0} + \md B_u^T M_u^{-1} \tilde p_u(t)
}{x3}
Taking together \cref{x1}, \cref{x2}, \cref{x3}, we have 
\eq{
 \ub{\bp \dot{\tilde p}_d \\ \dot{\tilde p}_u \\ \dot{\tilde q} \ep}{\dot{\tilde z}} &= \ub{\bp -R_d M_d^{-1} & 0 & -\md B_dW 
\\ 0 & -R_uM_u^{-1} & -\md B_uW \\ \md B_d^TM_d^{-1} & \md B_u^TM_u^{-1} & 0  \ep}{A} \ub{\bp {\tilde p}_d \\ {\tilde p}_u \\ {\tilde q} \ep}{{\tilde z}} 
}
\qed\end{proof}

$\mc S^{LS}$ can also be written as the unobservable subspace of the reduced system that gives the dynamics of the undamped nodes written in the $(\tilde p, \tilde q)$ coordinates: $\mc S^{LS} = \breve Q \ker(Obs(\breve C, \breve A))$, with
\eq{
\breve Q= \bp 0_{n_dr \times n_ur} & 0_{n_dr \times mr} \\ I_{n_ur} & 0_{n_ur \times mr} \\ 0_{mr \times n_ur} & I_{mr} \ep,
}
\eq{
\breve C =  \bp   0_{n_dr \times n_ur} & \md B_d W \\ R_uM_u^{-1} & 0_{n_ur \times mr} \\ 0_{(m-n+1)r \times n_ur} & \md C^T  \ep, \; \;
\breve A = \bp 0_{n_ur \times n_ur} & -\md B_uW \\ \md B_u^TM_u^{-1} & 0_{mr \times mr} \ep
}
Here, the row and column decomposition of $\breve A$ is such that the upper rows and left columns are associated with $\tilde p_u$ and the bottom rows and right columns with $\tilde q$.  

\medskip
{\bf Remark } 
\bi
\item The state trajectories $col(\tilde y_u, \tilde s_u)$ of the system $(\hat C, \hat A)$ describe the behavior of the undamped nodes in the reduced graph with total Laplacian matrix $\tilde{\md L}_u$, which is obtained by eliminating the damped nodes according to a Kron reduction. Kron reduction changes the topology including edge weights, but connectivity is preserved, hence $\ker(\tilde{\md L}_u)=\im( \I_r )$. $\hat Q$ represents the transformation matrix of the $(\tilde y_u(t), \tilde s_u(t))$ coordinates to the $(\tilde p_d(t), \tilde p_u(t), \tilde q_u(t))$ coordinates at steady state. Furthermore, $\md L^i_i \tilde y_u=0$ is an algebraic constraint that boils down to $\md B_d W \tilde q=0$, i.e. zero net force at damped nodes in the original graph. Finally, the constraint $R_u \tilde y_u=0$ assures that partially undamped nodes can only move in directions in which they do not experience damping. 
\item The system $(\breve C, \breve A)$ gives the dynamics of the undamped nodes when the damped nodes would be fixed at a single position. The set of solutions of $(\breve C, \breve A)$ for which $\md B_dW \tilde q(t)\equiv 0$, $R_uM_u^{-1} \tilde p_u(t) \equiv 0$ and $\tilde q(t) \in \im(\md B^T)$ for all $t \geq 0$ is the orthogonal projection of the steady-state trajectories of $\tilde z$ of the system \cref{y2} defined on $\Omega$ onto the $\tilde p_u$ and $\tilde q$ coordinate space. This yields $\mc S^{LS}$ by taking $\tilde p_d(t) \equiv 0$. 
\ei
\medskip

We show that all solutions in $\mc S^{LS}$ are composed of sinusoids:

\begin{proposition}\label{nwg}  Each solution $\tilde z(t) \sbq \mc S^{LS}$ generated by system \cref{y2} can be written as a finite sum of sinusoids. 
\end{proposition}

\begin{proof} Note that $\breve A$ can be written as a product of a skew-symmetric matrix and a positive-definite diagonal matrix:

\[ \breve A= \underbrace{\bp 0 & -\md B_u \\ \md B_u^T  & 0 \ep}_{\breve B} \underbrace{\bp M_u^{-1} & 0 \\ 0 & W \ep}_{\breve W} \] 

Therefore, $\breve A= \breve B \breve W$ is similar to the real skew-symmetric matrix $\breve W^{\frac{1}{2}} \breve B \breve W^{\frac{1}{2}}$ and thus it has purely imaginary eigenvalues that are semisimple\footnote{A real skew-symmetric matrix is a normal matrix which has the property to be diagonalizable and therefore, its eigenvalues are semisimple.}. Consequently, for any solution $\tilde z(t) \sbq \mc S^{LS}$ of \cref{y2}, $\tilde p_d(t) \equiv 0$ and the trajectories $col(\tilde p_u(t), \tilde q(t) ) = e^{\breve{A} t} col(\tilde p_u(0),  \tilde q(0) )$ are composed of periodic functions of the form $\cos( \mathfrak I(\lambda_i) t)  \mathfrak R(x_i)+\sin( \mathfrak I(\lambda_i) t) \mathfrak I(x_i)$, where $x_i$ is an eigenvector of $\breve A$ and $\lambda_i$ the associated semisimple, purely imaginary eigenvalue. \qed\end{proof}

{\bf Remark } Each component $\cos( \mathfrak I(\lambda_i) t)  \mathfrak R(x_i)+\sin( \mathfrak I(\lambda_i) t) \mathfrak I(x_i)$ of $\tilde z(t)$ corresponds to a group of nodes that is oscillating with the same angular frequency $\mathfrak I(\lambda_i)$.  The nodes in this group are indicated by the non-zero entries in $x_i$. If undamped nodes belong to multiple oscillating groups, they might oscillate with multiple frequencies.  

Due to this periodic character of the components of the solutions, we cannot find a proper subset of $\mc S^{LS}$ to which all solutions converge. Hence,

\begin{corollary}\label{smallest} The smallest set to which all solutions of \cref{y2} in $\Omega$ converge is given by $\mc S^{LS}$. 
\end{corollary}

From the periodic character of $\tilde z(t)$ at steady state, the solutions of \cref{y2} in $\Omega$ are bounded. This has an important implication: the sum of the momenta of the undamped nodes turn out to be zero: 

\begin{proposition}\label{conservation} {(Conservation of momentum at steady state)} For any solution $\tilde z(t) = col( \tilde p_d(t), \tilde p_u(t), \tilde q(t) )$ of \cref{y2} in $\Omega$, it holds that $\I_r^T \tilde p_u(t) \equiv 0$. 
\end{proposition}

\begin{proof} Since $\I_r^T\md B=0$, it follows from \cref{tored} that $\I_r^T \dot{\tilde p}_u(t)=-\I_r^T \md B_uW \tilde q(t)=\I_r^T \md B_dW \tilde q(t) \equiv 0$ at steady state. Hence, $\I_r^T \tilde p_u(t) \equiv K$, for some constant vector $K \in \R^r$. Suppose for the sake of contradiction that $K \neq 0$. Decompose the solution $\tilde p_u(t)$ into components in $\im(\tilde {\md L}_u)$ and $\im(\I_r)$: $\tilde p_u(t) = \tilde{\md L}_u \tilde g(t) + \I_r K \fone{n_u}$, which must hold for some function $\tilde g(t)$. Then $\tilde q$ can be expressed as 
\eq{
 \tilde q(t)- \tilde q(0)\isen \int_0^t \dot{\tilde q}(\tau) d\tau= \int_0^t \md B_u^T M_u^{-1}{\tilde p}_u(\tau) d\tau \issa 
  \int_0^t \md B_u^T M_u^{-1}\tilde{\md L}_u \tilde g(\tau)d\tau + \md B_u^T M_u^{-1} \I_r K \frac{t}{n_u}
  }
Since $n_d, n_u \geq 1$, $\md B_u^T$ has full column rank: there exists a (partially) undamped node, which we label as $k$, that is adjacent to a damped node. Now, define $E_k=e_k \otimes I_r$, where $e_k$ denotes the $k$'th unit vector. Then $\ker(\md B_u^T)=\ker((\md B_u^u)^T) \cap \ker((\md B_u^i)^T)  \sbq \im(\I_r) \cap \ker(E_k^T) = \{0\}$. Thus, $\md B_u^T$ has indeed full column rank and as a consequence, $(\md B_u^T M_u^{-1})^\dagger \md B_u^T M_u^{-1}=I_{n_ur}$. But then $ \I_r^T(\md B_u^T M_u^{-1})^\dagger (\tilde q(t) - \tilde q(0))  = Kt $, which is unbounded and contradicts the boundedness of $\tilde q(t)$. Therefore, $K=0$ and the result follows. 
\qed\end{proof}

{\bf Remark } In the original coordinates we have $\I_r^T p_u(t) = \I_r^T (\tilde p_u(t) + M_u \I_r \beta) \equiv \I_r^T M_u \I_r \beta$, which is nonzero for $\beta \neq 0$. This is the reason to refer to this characteristic as conservation of momentum rather than zero net momentum. 

In the $(\tilde y_u, \tilde s_u)$ coordinates, we find that conservation of momentum leads to $\I_r^T M_u \tilde y_u(t) \equiv 0$. Thus, $\I_r^T M_u \tilde y_u(t)$ might serve as an additional output variable to the system $(\hat C, \hat A)$ that does not affect the unobservable subspace. What is more, the same holds for its integral $\I_r^T M_u \tilde s_u(t)$ so that $\mc S^{LS}$ can be written equivalently as follows:

\begin{corollary}\label{bestbelangrijk2} $\mc S^{LS}= \hat Q \ker(Obs(\doublehat{C}, \hat A))$, where
\eq{
\doublehat{C} \isen  \bp  \md L_i^i & 0 \\ R_u & 0 \\ 0 & \I_r^T M_u \ep
}
\end{corollary}

\begin{proof} $(\sbq)$ Due to the freedom to choose an $\tilde s(0)$ that satisfies \cref{qq}, we can choose one that satisfies $\I_r^T M_u \tilde s_u(0)=0$. To see that this is possible, consider solutions $\tilde s(t)=col( \tilde s_d(t),  \tilde s_u(t) )$ and $\tilde y_u(t)$ that satisfy \cref{ru}, \cref{qq}, \cref{sdd}, \cref{nul} and \cref{dyna}. Replacing $\tilde s(t)$ by $\tilde s^*(t) = \tilde s(t) - \I_r \alpha$ with 
$$ \alpha = (\I_r^TM_u \I_r)^{-1} \I_r^TM_u \tilde s_u(0)$$ 
preserves these identities and furthermore $\I_r^T M_u \tilde s_u^*(0)=0$. Since $\I_r^T M_u \dot{\tilde s}^*_u(t)= \I_r^T M_u \dot{\tilde s}_u(t)=\I_r^T \tilde p_u(t) \equiv 0$, it holds that $\I_r^T M_u \tilde s_u^*(t) \equiv 0$. 
$(\spq)$ This follows from \cref{bestbelangrijk}: $\hat Q \ker(Obs(\doublehat{C}, \hat A)) \sbq \hat Q \ker(Obs({\hat C}, \hat A)) = {\mc S}^{LS}$
\qed\end{proof}

\subsection{Conditions on output consensus}

By combining \cref{qook}, \ref{pede} and \cref{smallest},  \ref{bestbelangrijk2}, we find that the pair $(\doublehat C,\hat A)$ is observable if and only if output consensus is guaranteed:

\begin{proposition}\label{prop}
All output trajectories $ y(t)$ of $ \Sigma \times  \Gamma$ converge to a point in $\im\{ \I_r\}$ if and only if $\ker(Obs(\doublehat{C}, \hat A))=\{ 0 \}$.
\end{proposition}

\begin{proof} From \cref{qook}, all output trajectories $ y(t)$ of $ \Sigma \times  \Gamma$ converge to a point in $\im\{ \I_r\}$ if and only if every trajectory $\tilde z(t) = col ( \tilde p(t), \tilde q(t))$ of \cref{y2} defined on $\Omega$ converges to the origin. Since $\mc S^{LS}$ is the smallest set to which all state trajectories of \cref{y2} defined on $\Omega$ converge, that is equivalent to $\mc S^{LS}=\hat Q \ker(Obs(\doublehat C, \hat A))=\{ 0 \}$. It follows immediately that $\ker(Obs(\doublehat C, \hat A))=\{ 0 \}$ results in $\mc S^{LS}=\{0\}$. Now suppose that $\mc S^{LS}=\hat Q \ker(Obs(\doublehat C, \hat A))=\{0\}$, i.e. $\ker(Obs(\doublehat C,\hat A)) \sbq \ker(\hat Q)$. Note that $\ker(Obs(\doublehat C, \hat A)) \sbq \ker \bp 0 & \I_r^T M_u \ep$ and 
\eq{
\ker(\hat Q) \sbq & \ker  \haakk{ \bp  0 & M_u^{-1} & 0 \\ 0 & 0 & \md B_uW  \ep \hat Q } \issa \ker \bp I & 0 \\ 0 & \tilde{\md L}_u \ep = \im \bp 0 \\ \I_r \ep
}
Hence, 
\eq{
 & \ker(\hat Q) \cap \ker(Obs(\doublehat{C}, \hat A))  \\ \sbq & 
 \;  \im \bp 0 \\ \I_r \ep \cap \ker \bp 0 & \I_r^T M_u \ep=\{0\}
}
Then, by assumption we obtain $\ker(Obs(\doublehat{C}, \hat A))=\{0\}$. \qed\end{proof}

We come to the following equivalence relation that connects the output consensus problem with the eigenspaces of $M_u^{-1}\tilde{\md L}_u$ and $M^{-1}\md L$:

\medskip

\begin{theorem}\label{t1} The following is equivalent:

\begin{enumerate}[(i)]
\item Every plant output trajectory $y(t)$ of $\Sigma \times \Gamma$ converges to a point in the set $\im\{\I_r\}$.
\item None of the eigenvectors of $M_u^{-1} \tilde{\md L}_u$ is contained in the intersection of the kernel of $\md{L}_i^i$ and the kernel of $R_u$, i.e. for each $\mu \in \sigma(M_u^{-1} \tilde{\md L}_u)$:
\eqq{\ker(M_u^{-1} \tilde{\md L}_u- \mu I) \cap \ker \bp \md L_i^i \\ R_u \ep = \{ 0\}}{ii}
\item Every eigenvector of $M^{-1} \md L$ in the kernel of $R$ has at least one nonzero value in an entry that corresponds to a damped node, i.e. for each $\mu \in \sigma(M^{-1}\md L)$: 
\eqq{\ker(M^{-1}\md L- \mu I) \cap  \ker(R) \cap \im \bp 0 \\ I_{n_ur} \ep =\{0\}}{iii}
\end{enumerate}
\end{theorem}

\begin{proof} \textit{(i) $\iff$ (ii)} From \cref{prop}, condition (i) holds if and only if $\ker(Obs(\doublehat{C}, \hat A))=\{ 0 \}$. 
 According to Hautus lemma (\cref{hautus}), that is equivalent to 

\[ \x{rank} \bp \hat A - \lambda I \\ \doublehat{C} \ep = 2n_u \qquad  \forall \lambda \in \mathbb{C} \]

Equivalently, from the rank-nullity theorem: $\forall \lambda \in \mathbb{C}$ it must hold that if

\begin{equation}\label{equagon} 
\bp -\lambda I & -M_u^{-1} \tilde{\md L}_u \\ I & - \lambda I\\  \md L_i^i & 0 \\ R_u & 0 \\ 0 & \I_r^TM_u \ep 
\bp  \tilde y_u \\ {\tilde{s}}_u \ep = 0 
\end{equation}

then $\tilde y_u=0$ and $\tilde{s}_u =0$. For $\lambda=0$, this implication always holds, since from the second block row it follows that ${\tilde y}_u=0$ and from the first block row, $\tilde s_u \in \im\{ \I_r \}$, which, combined with the bottom block row $\I_r^T M_u \tilde s_u=0$, yields $\tilde s_u=0$ (notice that $ \im( \I_r ) \cap \ker( \I_r^T M_u) =\{0\}$). 

So it remains to consider $\lambda \in \C \backslash \{0\}$, for which ${\tilde  y}_u= \lambda \tilde s_u$. Inserting this in the first block row yields $\lambda^2 {\tilde s}_u = - M_u^{-1} \tilde{\md L}_u  {\tilde s}_u$. Premultiplying by $\fone{\lambda^2} \md I_r^T M_u$ yields $\I_r^TM_u {\tilde s}_u =  -\frac{1}{\lambda^2} \I_r^T M_u M_u^{-1} \tilde{\md L}_u {\tilde y}_u=0$, hence the last block row is always satisfied if the block rows above are satisfied too. Thus, for all $\lambda \neq 0$, the only solutions of \cref{equagon} are $\tilde y_u=0$, $\tilde s_u=0$ if and only if for all $\lambda \neq 0$,
\eqq{
 M_u^{-1}  \tilde{\md L}_u { \tilde s}_u  \isen \; -\lambda^2 {\tilde s}_u \\ 
\lambda \md L_i^i {\tilde s}_u \isen \; 0 \\
\lambda R_u \tilde s_u \isen \; 0
}{kkkw}

implies ${\tilde s}_u=0$ (and therefore also $\tilde y_u={\lambda} \tilde s_u=0$). That is, for any eigenvalue $\mu=-\lambda \in \sigma(M_u^{-1}\tilde{\md L}_u) \backslash \{0\}$, \cref{ii} holds. Since for $\mu=0 \in \sigma(M_u^{-1}\tilde{\md L}_u)$, $\ker(M_u^{-1} \tilde{\md L}_u - \mu I) = \im(\I_r)$ and $\im(\I_r) \cap \ker(\md L_i^i)=\{0\}$\footnote{This holds since $\md L_i^i$ is a nonzero and nonpositive matrix}, \cref{ii} always holds for $\mu=0$.

\textit{(ii) $\iff$ (iii)} Write out $\tilde{\md L}_u$ in the left-hand side of the first equation in \cref{kkkw} where $\lambda \neq 0$ and use the second constraint to obtain: 
\eq{
M_u^{-1}  \tilde{\md L}_u {\tilde s}_u  &= M_u^{-1} (\md L_u^u+\md L_u^i -  (\md L_i^i)^T(\md L_d^d+\md L_d^i)^{-1}\md L_i^i) {\tilde s}_u \\
&=  M_u^{-1}(\md L_u^u+\md L_u^i) {\tilde s}_u 
}
With this and the fact that $\ker(\md L_i^i)=\ker(M_d^{-1}\md L_i^i)$, the first two identities in \cref{kkkw} are equal to 
\eq{ 
\bp M_d^{-1} & 0 \\ 0 & M_u^{-1} \ep \bp \md L_d^d+\md L_d^i & \md L_i^i \\  (\md L_i^i)^T & \md L_u^u+ \md L_u^i \ep \bp 0 \\ {\tilde s}_u \ep &=  - \lambda^2 \bp 0 \\ \tilde s_u \ep
}
The last identity in \cref{kkkw} can be rewritten as
\eq{
\bp R_d & 0 \\ 0 & R_u \ep \bp 0 \\ \tilde s_u \ep = \bp 0 \\ 0 \ep
}
Hence, \cref{ii} is true for any $\mu \in \sigma(M_u^{-1} \tilde{\md L}_u) \backslash \{0\}$  if and only if there does not exist an eigenvector of $M^{-1}\md L$ corresponding to a nonzero eigenvalue that is in the kernel of $R$ and of the form $col( 0, \; {\tilde s}_u)$. Also, \cref{iii} is always true for $\mu=0$. Thus, the latter condition is equivalent to (iii).
\qed\end{proof}


We give the following corollary without proof:  

\begin{corollary}
By a change of coordinates, we can extend \cref{t1} with the following equivalent conditions: 
\begin{enumerate}[(i)]
\setcounter{enumi}{3}
\item None of the eigenvectors of $\tilde{\md L}_u M_u^{-1}$ is contained in the intersection of the kernels of $\md L_iM_u^{-1}$ and $R_u M_u^{-1}$. 
\item Every eigenvector of $\md L M^{-1}$ in the kernel of $RM^{-1}$ has at least one nonzero value in an entry that corresponds to a damped node.
\end{enumerate}
\end{corollary}

\section{Topological conditions for parameter-independent stability}\label{sec5} 
$\vphantom{1}$\\ The equivalence relation in \cref{t1} formulates a matrix-theoretic approach to check if a system achieves output consensus based on its parameters $W$, $M$, $R$ and the graph structure. In this section, we assume that we only know the graph structure and the set of damped nodes and show how to extend the result of \cref{t1} to solve the parameter-independent stability problem as formulated in \cref{secprobleem}. Motivated by the observation that asymptotic output consensus of system $\cref{y}$ is equivalent to GAS of system $\cref{y2}$ defined on $\Omega$ (\cref{qook}), we define the notion of parameter-independent global asymptotic stabiltity:

\begin{definition}
Given a damping graph $\mc G=(V,V_d,E)$. If the system \cref{y2} defined on $\Omega$ is GAS for all feasible system parameters $M$, $W$ and $R$, where the structure of $R$ is according to the set $V_d \sbq V$ of damped nodes, then this system is said to be parameter-independent globally asymptotically stable, abbreviated PI-GAS. 

\end{definition}
In the sequel we abuse terminology and say that the damping graph $\mc G$ is PI-GAS instead of the system \cref{y2}. The parameter-independent stability problem boils down to the question whether a given damping graph $\mc G=(V,V_d,E)$ is PI-GAS. Note that if $\mc G$ is not PI-GAS, there exist system parameter values and an initial condition that lead to nonzero oscillatory behavior. Notice further that turning undamped nodes into partially undamped nodes would never lead to a bigger set at the left-hand side of \cref{iii}. Therefore, for determinining PI-GAS of $\mc G$, we assume without loss of generality that the nodes in $V \backslash V_d$ are undamped. Also, we pose the following assumption on the feasibility of the system parameters: 
\begin{assumption}\label{assu2} For all nodes $i \in V$ and edges $k \in E$, the mass matrix $M_i$ and edge weight matrix $W_k$ are diagonal. Consequentely, $M$ and $W$ are diagonal and $\mc G_{\md L}$, the graph associated with the Laplacian matrix $\md L$, consists of $r$ connected components that are copies of $\mc G$. 
\end{assumption}

This assumption is natural: when the matrices $M_i$ and $W_k$ are diagonal, the velocity $y_i=M_i^{-1}p_i$ and force $f_k=W_kq_k$ are in the same orthant as $p_i$ and $q_k$, respectively.\footnote{The general case, where $M_k$ and $W_k$ are any positive-definite matrices, much likely gives rise to a more complex analysis as $\mc G_{\md L}$ depends on the structure of these matrices and moreover, negative edge weights are allowed. This generalization is left as a topic for future research.}

\cref{t1} can be extended to give necessary and sufficient conditions for which $\mc G$ is PI-GAS under the feasibility condition stated in \cref{assu2}.  

\begin{lemma}\label{kg}
The damping graph $\mc G$ is PI-GAS if and only if $\tilde K_{\mc G}=\{0\}$, where 
\end{lemma}
\eqn{\tilde K_{\mc G}:=\bigcup_{\tilde M, \tilde W \in \tilde \Lambda} \ker \haakk{  B \tilde W B^T - \tilde M }  \cap \im \bp 0_{n_d \times n_u} \\ I_{n_u} \ep 
}{kaagee}
\textit{and $\tilde \Lambda$ is the set of $n \times n$ diagonal positive-definite matrices.}

\begin{proof} Let $\Lambda$ be the set of $nr \times nr$ positive-definite diagonal matrices. Consider any system with dynamics \cref{y2} defined on $\Omega$ with underlying damping graph $\mc G$ and system parameters $M_0, W_0 \in \Lambda$ and $R_0=diag(R_d,R_u)$ . By the equivalence relation $(1) \iff (3)$ of \cref{t1} and \cref{qook}, this system is GAS if and only if for all $\mu \in \sigma(M_0^{-1} \md BW_0 \md B^T)$, we have
$$ \ker(M_0^{-1} \md BW_0 \md B^T- \mu I) \cap \ker \bp R_d & 0 \\ 0 & R_u \ep \cap \im \bp 0_{n_dr \times n_ur} \\ I_{n_ur} \ep = \{0\}$$

This statement still holds if we replace `for all $\mu \in \sigma(M_0^{-1} \md BW_0 \md B^T)$' by `for all $\mu > 0$'. Indeed, all eigenvalues of $M^{-1}\md L$ are nonnegative and by the assumption that there is at least one damped node, $\ker(M^{-1}\md L)=\im(\I_r)$, whose intersection with $\im \bp 0_{n_ur \times n_dr} & I_{n_ur} \ep^T$ is $\{0\}$ and thus $\mu=0$ won't be a troublemaker either. For determining PI-GAS, we consider w.l.o.g. the worst-case scenario for the damping matrix, which is $R_u=0$. This gives $\ker(R) = \im \bp 0 & I_{n_ur} \ep^T$. Thus, $\mc G$ is PI-GAS if and only if for any inertia matrix $M \in \Lambda$, weight matrix $W \in \Lambda$ and scalar $\mu>0$, it holds that
$$\ker(M^{-1}\md BW \md B^T - \mu I) \cap \im \bp 0_{n_dr \times n_ur} \\ I_{n_ur} \ep = \{0\}$$
Since for any $\mu >0$ and $M \in \Lambda$,  $\ker(M^{-1}\md BW \md B^T- \mu I) = \ker(\md BW \md B^T-\mu M)$ and $\{ \mu M \mid \mu >0, \; M \in \Lambda \} = \Lambda$, we can take $\mu=1$ without loss of generality. Hence, $\mc G$ is PI-GAS if and only if 
\eqn{
K_{\mc G}:= \bigcup_{M, W \in \Lambda} \ker(\md BW \md B^T - M) \cap \im \bp 0_{n_dr \times n_ur} \\ I_{n_ur} \ep = \{0\}
}{kaagee2}
Notice the differences in dimensions of $K_{\mc G}$ ($nr\times nr$) and $\tilde K_{\mc G}$ ($n \times n$). To see that $\tilde K_{\mc G}=\{0\}$ is a necessary and sufficient condition for PI-GAS, observe that $\md L=\md B W \md B^T$ is reducible under \cref{assu2}: it is easily checked that $\md L_{ij}=0$ if $mod(i-j,r) \neq 0$. This allows us to write $K_{\mc G}$ as an intersection of $r$ equally sized subspaces associated with each dimension. Define the $nr \times n$ matrix $P_i:=\bp e_i & e_{i+r} & \dots & e_{i+(n-1)r}\ep$ and $mr \times m$ matrix $Q_i:=\bp e_i & e_{i+r} & \dots & e_{i+(m-1)r}\ep$, where $e_j$ denotes the $j$'th unit vector and $i \in \{1, \dots, r\}$. By reducibility of $M$ and $\md L$, we have that $[P_k^T(\md B W \md B^T -M)]_{ij}=0$ if $mod(j,r) \neq k$. Since also 
$$[P_kP_k^T]_{ij}=\bc 1 & i=j \wedge mod(i,r)=k \\ 0 & \x{otherwise,} \ec$$  
we have that $P_i^T(\md B W \md B^T -M)=P_i^T(\md B W \md B^T -M)P_iP_i^T$. 
This gives
\eq{ 
\ker \haak{ \md BW \md B^T - M } \isen \bigcap_{i=1}^r \ker \haakk{ P_i^T \haak{ \md BW \md B^T - M }} \issa
\bigcap_{i=1}^r \ker \haakk{ P_i^T \haak{ \md BW \md B^T - M } P_i P_i^T} 
}
Similarly, $[P_k^T \md B]_{ij}=0$ if $mod(j,r)\neq k$ and $P_i^T \md B =P_i^T \md B Q_i Q_i^T$. Since furthermore, $B=P_i^T\md B Q_i$, we have that $P_i^T \md B = B Q_i^T$ and hence
\eq{ 
\ker \haak{ \md BW \md B^T - M } \isen \bigcap_{i=1}^r \ker \haakk{ \haak{B Q_i^T W Q_i B^T - P_i^T M P_i}P_i^T }}
We obtain for $K_{\mc G}$: 
\begin{align}\label{eqa}
K_{\mc G} \isen \hspace{-2.95mm} \bigcup_{M, W \in \Lambda} \bigcap_{i=1}^r \ker \haakk{ \haak{ BQ_i^TWQ_i B^T - P_i^T M P_i} P_i^T }  \cap \im \bp 0_{n_dr \times n_ur} \\ I_{n_ur} \ep 
\end{align}

We are now ready to show that $\tilde K_{\mc G}=\{0\}$ if and only if $K_{\mc G}=\{0\}$. Suppose first that $\tilde K_{\mc G}=\{0\}$. Consider any nonzero $w \in \im \bp 0_{n_ur \times n_dr} & I_{n_ur} \ep^T$ and any $M, W \in \Lambda$. There exists $i_0 \in \{1, \dots, r\}$ for which $P_{i_0}^Tw$ is nonzero. Since $\tilde K_{\mc G}=\{0\}$, $Q_{i_0}^TWQ_{i_0}, \: P_{i_0}^TM P_{i_0} \in \tilde \Lambda$ and $P_{i_0}^Tw \in  \im \bp 0_{n_u \times n_d} & I_{n_u} \ep^T$, which is nonzero, we have $(BQ_{i_0}^TWQ_{i_0}B^T - P_{i_0}^TM P_{i_0}) P_{i_0}^Tw \neq 0$. Consequently, $$w \notin \bigcap_{i=1}^r \ker \haakk{ \haak{ BQ_i^TWQ_i B^T - P_i^T M P_i} P_i^T }.$$ Since $w \in \im \bp 0_{n_ur \times n_dr} & I_{n_ur} \ep^T \backslash \{0\}$ and $M,W \in \Lambda$ were taken arbitrarily, $K_{\mc G}=\{0\}$.  

Now, suppose that $K_{\mc G}=\{0\}$. Consider any nonzero $v \in \im \bp 0_{n_u \times n_d} & I_{n_u} \ep^T$ and any $\tilde M, \tilde W \in \tilde \Lambda$. Define $M^*:= \sum_{i=1}^r P_i \tilde M P_i^T$ and $W^* := \sum_{i=1}^r Q_i \tilde W Q_i^T$ and notice that both matrices are in $\Lambda$. Also, let $w:=\sum_{i=1}^r P_iv$, which is nonzero. For all $i=1,\dots,r$, $P_i^T M^* P_i=\tilde M$, $Q_i^T W^* Q_i^T = \tilde W$ and 

$$ P_{i}^Tw=P_{i}^T \sum_{i=1}^r P_iv = \ub{P_{i}^T P_{i}\vphantom{\sum_{j \neq i} P_{i}^T}}{=I_n} v + \ub{\sum_{j \neq i} P_{i}^T P_j}{=0} v=v,$$

Since $w \in \im \bp 0_{n_dr \times n_ur} & I_{n_ur} \ep^T$, $M^*, \: W^* \in \Lambda$ and $K_{\mc G}=\{0\}$, we deduce that $w \notin \bigcap_{i=1}^r \ker \haakk{ \haak{ BQ_i^TW^*Q_i B^T - P_i^T M^* P_i} P_i^T }$. 
But $(BQ_i^TW^*Q_iB^T-P_i^TM^*P_i)P_i^Tw=(B\tilde WB^T-\tilde M)v$ is constant for all $i=1, \dots, r$, hence $(B\tilde WB^T-\tilde M)v \neq 0$. Since $v \in \im \bp 0_{n_u \times n_d} & I_{n_u} \ep^T \backslash \{0\}$ and $\tilde M, \tilde W \in \tilde \Lambda$ were taken arbitrarily, $\tilde K_{\mc G}=\{0\}$. 
\qed \end{proof}
 
\cref{kg} shows that determining PI-GAS of $\mc G$ is independent of the agent and controller state dimension $r$. Note further that $\tilde K_{\mc G}$ is indeed independent of the system parameters due to the union that is taken over all feasible system parameter values. The next lemma shows an easy test to verify whether a vector $v \in \R^n$ is contained in $\tilde K_{\mc G}$. 

\begin{lemma}\label{o} For any node $i \in V$, denote by $\mc N_i$ the set of neighbors of $i$ in the graph $\mc G$. A vector $v \in \R^{n}$ is contained in $\tilde K_{\mc G}$ if and only if the following conditions hold: 
\begin{enumerate}
\item $v_i=0$ for all indices $i$ of $v$ associated with damped nodes.
\item for all indices $i$ with $v_i=0$, there exists $j \in \mc N_i$ for which $v_j<0$ if and only if there exists $k \in \mc N_i$ for which $v_k>0$. 
\item for all indices $i$ with $v_i>0$, there exists $j\in \mc N_i$ for which $v_j<v_i$. 
\item for all indices $i$ with $v_i<0$, there exists $j \in \mc N_i$ for which $v_j>v_i$. 
\end{enumerate}
\end{lemma}

\begin{proof}
$(\Rightarrow)$ Suppose that $v \in \tilde K_{\mc G}$. Then it follows immediately that condition \textit{1} holds. Since $v \in \tilde K_{\mc G}$, $(B \tilde  W B^T-\tilde M)v=0$ for arbitrary $\tilde M, \tilde W \succ 0$. Consider the $i$'th row of $B \tilde W B^T-\tilde M$. It has negative entries on those columns $j$ for which $j \in \mc N_i$ and is zero on entries $j \notin \mc N_i \cup \{i\}$. Observering that $L_{ii}=(B \tilde W B^T)_{ii}=-\sum_{k=(i,j), \: j \in \mc N_i} \tilde W_{kk}$, we derive that 
\eq{
v_i \haak{ \sum_{\substack{k=(i,j) \in E }} \tilde W_{kk} -  \tilde M_{ii}}= \sum_{\substack{k=(i,j) \in E }}  v_j \tilde W_{kk}
}
Or equivalently,
\eqq{
v_i \tilde M_{ii} = \sum_{\substack{k=(i,j)\in E \\ v_j < v_i}}  |v_i-v_j| \tilde W_{kk} -  \sum_{\substack{k=(i,j)\in E \\ v_j > v_i}}  |v_i-v_j|\tilde W_{kk}
}{balanciones}

From this equality, conditions \textit{2}, \textit{3} and \textit{4} are immediately derived. 

$(\Leftarrow)$ Now suppose that a vector $v \in \R^n$ satisfies conditions \textit{1}, \textit{2}, \textit{3} and \textit{4}. Since the entries of $v$ corresponding to the damped nodes are zero by \textit{1}, it remains to show the existence of matrices $\tilde M \succ 0$ and $\tilde W \succ 0$ that satisfy $(B \tilde W B^T - \tilde M)v=0$. For each $i \in V$, the vector $v$ partitions its neighbors into 3 sets: 
\eq{
\mc N^{+}_i \isen \{ j \in \mc N_i \mid v_j \geq |v_i| \}, \\
\mc N^{0}_i \isen \{ j \in \mc N_i \mid |v_j| < |v_i| \}, \\
\mc N^{-}_i \isen \{ j \in \mc N_i \mid v_j \leq - |v_i| \}
}
The terms in the right-hand side of \cref{balanciones} can be grouped according to this partition: 
\eqq{
v_i \tilde M_{ii} = L_i^- - L_i^+ + \epsilon_i
}{balance2}
where
\eq{
L_i^{\pm} :\isen \hspace{-1mm}  \sum_{\substack{k=(i,j)\in E \\ j \in \mc N^{\pm}_i}}  \hspace{-1mm} |v_i-v_j|  \tilde W_{kk},  &
\epsilon_i :\isen \sum_{\substack{k=(i,j) \in E \\ j \in \mc N_i^0 }} (v_i-v_j) \tilde W_{kk}
}
Now, $v \in   \tilde K_{\mc G}$ if and only if for each $i=1,\hdots, n$ there exists a scalar $\tilde M_{ii}>0$ and scalars $\tilde W_{kk}>0$ for all $k=(i,j) \in E$ such that \cref{balance2} holds. 

One way to construct the matrix $\tilde W$ iteratively is as follows. Label the nodes such that $j>i$ whenever $|v_j|>|v_i|$. Start with $i=1$ according to this labelling and consider all edges $k=(i,j)$. Set each diagonal entry $\tilde W_{kk}$, corresponding to the edge $k=(i,j)$ with $j>i$ to be
\eq{
& \tilde  W_{kk} = \bc
1 & v_j=v_i \\
\fone{ |v_i-v_j| | \mc N^+_i|} & 0 \leq v_i < v_j , \; \epsilon_i=0 \\
\fone{ |v_i-v_j| | \mc N^-_i|} & v_j < v_i \leq 0, \; \epsilon_i=0 \\
\frac{2}{ |v_i-v_j| | \mc N^+_i|} & v_i < 0 < v_j, \; \epsilon_i=0 \\
\frac{2}{ |v_i-v_j| | \mc N^-_i|} & v_j < 0 < v_i, \; \epsilon_i=0 \\
\frac{\fone{2} |\epsilon_i|}{ |v_i-v_j| | \mc N^{+}_i|} & 0 \neq v_i < v_j, \; \epsilon_i \neq 0 \\
\frac{\fone{2} |\epsilon_i|}{ |v_i-v_j| | \mc N^{-}_i|} & v_j < v_i \neq 0, \; \epsilon_i \neq 0 \\
\ec
}
Repeat the assignment of values to the $\tilde W_{kk}$s iteratively for higher $i$. Note that with the chosen labelling, $\tilde W_{kk}$ is well-defined since $\epsilon_i$ does not depend on $\tilde W_{kk}$ values for edges $k$ that connect $i$ with nodes that have a higher index. Substituting these positive values of $\tilde W_{kk}$ in \cref{balance2}, we make the following observations: 
\smallskip
\bi
\item if $v_i=0$, then $\mc N_i^0=\emptyset$ and hence $\epsilon_i=0$. From condition \textit{2} it follows that $\mc N_i^+$ and $\mc N_i^-$ are either both empty or both nonempty. Hence, $(L_i^+,L_i^-,\epsilon_i) \in \{ (0,0,0), (1,1,0) \}$, yielding $v_i \tilde M_{ii}=0$. 
\item if $v_i>0$ and $\epsilon_i=0$, then again $\mc N_i^0=\emptyset$. From condition \textit{3}, it follows that there exists $j \in \mc N_i^-$, yielding $L_i^-=2$. Depending on whether $\mc N_i^+$ is empty or not, $L_i^+ \in \{1,0\}$, so we obtain $v_i \tilde M_{ii} \in \{1,2\}$. 
\item if $v_i<0$ and $\epsilon_i=0$, then similarly, $v_i \tilde M_{ii} \in \{-1,-2\}$. 
\item otherwise, $v_i \neq 0$ and $\epsilon_i \neq 0$. Depending on whether $\mc N_i^+$ and $\mc N_i^-$ are empty or not, $L_i^+, L_i^- \in \acco{ 0,\half |\epsilon_i| }$. This gives $v_i \tilde M_{ii} \in  \acco{ \half \epsilon_i, \epsilon_i, \frac{3}{2} \epsilon_i }$
\ei
\smallskip
Noting that $\epsilon_i \geq 0$ if $v_i>0$ and $\epsilon_i \leq 0$ if $v_i<0$, we see that in  all cases, \cref{balance2} is satisfiable for some $\tilde M_{ii}>0$. Therefore, $w \in  \tilde K_{\mc G}$.
\qed\end{proof}

By the above lemma, PI-GAS of $\mc G$ is independent of the edges between damped nodes.
Note that the conditions in \cref{o} can be completely rephrased in terms of the signs of $v \in \R^n$ and $B^Tv \in \R^m$, where the latter vector contains the entry differences of neighboring nodes. This makes explicit computing of $ \tilde K_{\mc G}$ a finite dimensional problem: it suffices to perform a brute-force search using (sign) vectors in $\{-1,0,1\}^{n+m}$. 

Interestingly, for determining whether $ \tilde K_{\mc G}=\{0\}$, only the first two conditions of \cref{o} are relevant. That is, there exists a nonzero vector that only meets the first two conditions if and only if there exists a nonzero vector that meets all four conditions, which holds if and only if $\mc G$ is not PI-GAS. This is shown in the following proposition: 

\begin{proposition}\label{t2} $\mc G$ is PI-GAS if and only if there does not exist a nonzero $v \in \R^n$ that satisfies the following two conditions:
\begin{enumerate}
\item $v_i=0$ for all indices $i$ of $v$ associated with damped nodes.
\item for all indices $i$ with $v_i=0$, there exists $j \in \mc N_i$ for which $v_j<0$ if and only if there exists $k \in \mc N_i$ for which $v_k>0$. 
\end{enumerate}
\end{proposition}

\begin{proof} $(\Rightarrow)$ Suppose that there exists a nonzero $v \in \R^n$ that meets conditions \textit{1} and \textit{2}. We construct a vector $w \in \R^n$ that meets the four conditions of \cref{o}. Denote by $dist_0(i)$ the minimum number of edges in a path in $\mc G$ from node $i$ to a node $j$ for which $v_j=0$. For all $i=1, \dots, n$, set $w_i=0$ if $v_i=0$, $w_i=dist_0(i)$ if $v_i>0$ and $v_i=-dist_0(i)$ if $v_i<0$. Obviously, $sgn(w)=sgn(v)$, so $w$ satisfies conditions \textit{1} and \textit{2}. Now, consider $i \in V$ for which $w_i>0$ and the path of minimum length that connects $i$ with a node $j \in V$ for which $w_j=0$. Then the second node on this path, say node $k$, satisfies $w_k=w_i-1$ and thus $w$ meets condition \textit{3} too. Similar for condition \textit{4}. Hence, by \cref{o}, $w \in  \tilde K_{\mc G}$. Again from $sgn(w)=sgn(v)$, $w$ is nonzero too and thus $ \tilde K_{\mc G} \neq \{0\}$, which, by \cref{kg}, implies that $\mc G$ is not PI-GAS. 

$(\Leftarrow)$ Suppose that $\mc G$ is not PI-GAS. Then by \cref{kg} and \ref{o}, there exists a nonzero vector $v \in \R^n$ satisfying the four conditions stated in \cref{o} and in particular the two conditions stated above. \qed\end{proof}

The above proposition can be completely rephrased in terms of the sign of $v$ and therefore, a brute-force search with vectors in $\{-1,0,1\}^n$ suffices to determine if $\tilde K_{\mc G}=\{0\}$. Thus, the parameter-independent stability problem is reduced to an even smaller finite dimensional problem. In fact, the problem is equivalent to a topological coloring problem, where each node $i$ that is black, red or blue corresponds to the value $v_i=0$, $v_i>0$ or $v_i<0$, respectively. Before we state the problem, we introduce the following terminology:

Consider a simple, connected damping graph $\mc G=(V,V_d,E)$, whose nodes are colored black, blue and red. A black node is called \textit{poorly balanced} if all its neighbors are black, \textit{richly balanced} if it has at least one red and at least one blue neighbor and is \textit{unbalanced} otherwise.

Also, the colored graph of $\mc G$ is called \textit{poorly balanced} if $\mc G$ consists entirely of black nodes, \textit{richly balanced} if $\mc G$ contains at least one richly balanced black node and no unbalanced black nodes. Otherwise, $\mc G$ is called \textit{unbalanced} and contains at least one unbalanced black node.

\medskip

{\bf Problem} \textit{(Richly balanced coloring (RBC) problem)} Given a damping graph $\mc G$, determine if there exists an assignment of colors that renders $\mc G$ richly balanced.

\begin{definition} In a damping graph $\mc G=(V,V_d,E)$, a node $v \in V$ is said to be \textit{uncolorable} if $v$ is black in every balanced colored graph of $\mc G$. The set of uncolorable nodes is denoted by $\mc U(\mc G)$. All other nodes are called \textit{colorable} nodes.\end{definition}

Using the RBC problem formulation and the above definition we restate \cref{t2} in terms of richly balanced colorings and uncolorable nodes: 

\begin{proposition}\label{t3} Consider a damping graph $\mc G=(V,V_d,E)$. The following statements are equivalent:
\begin{enumerate}[(i)]
\item $\mc G$ is PI-GAS. 
\item There does not exist an assignment of colors that renders $\mc G$ richly balanced. 
\item $\mc U(\mc G)=V$, i.e. every node in $\mc G$ is uncolorable.
\end{enumerate}
\end{proposition}

\subsection{Time complexity of the RBC problem}

Unfortunately, the RBC problem has a high time complexity as it is NP-complete.  A problem is NP-complete if every solution can be checked for feasibility in polynomial time (i.e. it is in NP) and furthermore, it is at least as hard as the hardest problems in NP (i.e. it is NP-hard). It is clear that the RBC problem is in NP. We show that the RBC problem is NP-hard by a reduction from SAT, a problem which is known to be NP-hard. 

\begin{theorem}\label{nphard} The decision problem whether a given damping graph $\mc G=(V,V_d,E)$ is PI-GAS is NP-complete. 
\end{theorem}

\begin{proof} The SAT problem is a decision problem that asks if we can assign to $n$ boolean variables the value 'true' or 'false' such that a given boolean formula that includes these variables is true. Here, the boolean formula is assumed to be a conjunction of disjunctive clauses, where each clause consists of positive or negative literals associated with the $n$ boolean variables.\footnote{A positive literal is a boolean variable; a negative literal is its negation  A disjunctive clause is an expression of a finite collection of positive and negative literals that is true if and only if at least one literal is true. A conjunction of clauses is true if and only if all clauses are true.} We will show now how we can use the RBC problem to solve an instance $\mc I$ of the SAT problem. $\mc I$ is defined by a boolean formula consisting of $p$ clauses and containing $n$ boolean variables. We construct the connected damping graph $\mc G_{\mc I}=(V,V_d,E)$ in the following way: 

\smallskip

\bi
\item Create an undamped base node, denoted by $v_0$. 
\item For each boolean variable $x_i$, create 2 undamped nodes: $v_i^+$ and $v_i^-$ and a damped node $v_i^0$. Create edges $(v_i^+, v_i^0)$ and $(v_i^0, v_i^+)$. For $i=1$ create the edge $(v_0, v_1^-)$, for $i \geq 2$ create an edge $(v_{i-1}^+, v_i^-)$. 
\item For each clause $j=1,\dots, p$, create a damped node $\bar v_j$ and an edge $(v_0, \bar v_j)$. Also, for each positive literal $x_i$ in clause $j$, create an edge $(\bar v_j, v_i^+)$. For each negative literal $\neg x_i$ in clause $j$, create an edge $(\bar v_j, v_i^-)$.  
\ei

\smallskip

Note that $\mc G_{\mc I}$ is connected and $\mc I$ can be derived uniquely from $\mc G_{\mc I}$ and the labelling of the nodes. We show that the following statements are equivalent: 

\smallskip

\begin{enumerate}
\item The boolean formula of $\mc I$ is satisfiable, i.e. there exists an assignment of truth and false values to $x_1, \dots, x_n$ such that the boolean formula of $\mc I$ is true
\item There exists a richly balanced colored graph of the damping graph $\mc G_{\mc I}$. 
\end{enumerate} 

\smallskip

Consider the color function $f_c$ that maps an assignment of boolean variables to a coloring of $\mc G$ as follows:

\eq{ 
f_c(v_0) \isen \x{red} \\ f_c(v_i^+) \isen \bc \x{blue} & x_i \x{ is true} \\ \x{red} & x_i \x{ is false} \ec \\ f_c(v_i^-) \isen \bc \x{blue} &  x_i \x{ is false} \\ \x{red} & x_i \x{ is true} \ec 
}

$(1) \implies (2)$ Consider an assigment of true and false values such that the boolean formula is true. Apply the color function $f_c$ to this assignment to obtain a colored graph of $\mc G_{\mc I}$. Each node of the type $v_i^0$ in the colored graph is richly balanced since it has a red and blue neighbor. Also, each $\bar v_j$ is richly balanced since it has a red neighbor ($v_0$) and at least one blue neighbor, which is the node corresponding to the literal that is true in clause $j$. There are no other black nodes and consequently the created colored graph is richly balanced.

$(2) \implies (1)$ $\mc G_{\mc I}$ satisfies the property that if one of the undamped nodes is black in a balanced colored graph of $\mc G_{\mc I}$, then all undamped nodes are black. This can be easily shown by applying the zero forcing algorithm. Also, by the red-blue symmetry of the RBC problem, (2) is true if and only if there exists a richly balanced colored graph in which $v_0$ is red. Consider such a richly balanced colored graph of $\mc G_{\mc I}$, then every $v_i^0$ must be richly balanced and thus precisely one of its neighbors $v_i^+$ and $v_i^-$ is red; the other one is blue. Note that the color function $f_c$ according to which $\mc G_{\mc I}$ is colored, has a unique inverse, which is the corresponding assignment of the boolean variables. Every node $\bar v_j$ has at least one blue neighbor, which corresponds to the literal which is true in clause $j$. Therefore, the boolean formula corresponding to ${\mc I}$ is true. 

If there exists an algorithm that solves the RBC problem in polynomial time, then any instance $\mc I$ of SAT can be solved in polynomial time too. Indeed, given $\mc I$, we create the damping graph $\mc G_{\mc I}$ in time $\mc O(|\mc I|)$, solve the RBC  problem with the polynomial algorithm and if the output is 'yes', the boolean formula of $\mc I$ is satisfiable. Hence, if there does not exists an algorithm that solves SAT in polynomial time (this question boils down to the P versus NP problem, which is a Millennium Prize Problem), then there does not exist such an algorithm for the RBC  problem either. Therefore, the RBC problem is as hard as SAT and consequently, it is NP-hard. 
\qed\end{proof}

{\bf Example}

Consider an instance $\mc I$ of SAT defined by the boolean formula $(\neg x_1 \vee x_2) \wedge (\neg x_1 \vee \neg x_2) \wedge (x_1 \vee \neg x_3 \vee \neg x_4)$. This formula is satisfiable, e.g. with $x_1$ and $x_4$ being false and $x_2$ and $x_3$ being true. The corresponding richly balanced colored graph of $\mc G_{\mc I}$ is illustrated in Figure 1. 

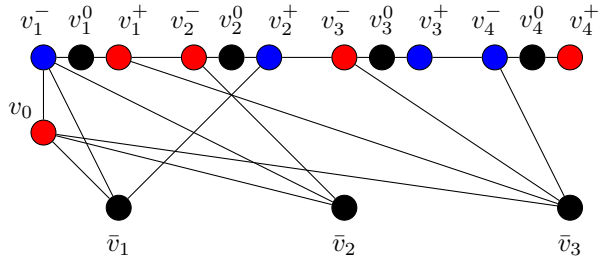
\begin{figure}[h!]
\begin{center}
\begin{tikzpicture}
[
b/.style={circle, draw=black,fill=black},
w/.style={circle, draw=black, fill=white},
bl/.style={circle, draw=black, fill=blue},
r/.style={circle, draw=black, fill=red}
]

\node[r] (base) at (1,0) {};
\node[bl] (1m) at (1,1) {};
\node[b] (1) at (1.5,1) {};
\node[r] (1p) at (2,1) {};
\node[r] (2m) at (3,1) {};
\node[b] (2) at (3.5,1) {};
\node[bl] (2p) at (4,1) {};
\node[r] (3m) at (5,1) {};
\node[b] (3) at (5.5,1) {};
\node[bl] (3p) at (6,1) {};
\node[bl] (4m) at (7,1) {};
\node[b] (4) at (7.5,1) {};
\node[r] (4p) at (8,1) {};

\node[b] (c1) at (2,-1) {};
\node[b] (c2) at (5,-1) {};
\node[b] (c3) at (8,-1) {};

\node at (.7,.3) {$v_0$};
\node at (.9,1.5) {$v_1^-$};
\node at (1.5,1.5) {$v_1^0$};
\node at (2.2,1.5) {$v_1^+$};
\node at (2.9,1.5) {$v_2^-$};
\node at (3.5,1.5) {$v_2^0$};
\node at (4.2,1.5) {$v_2^+$};
\node at (4.9,1.5) {$v_3^-$};
\node at (5.5,1.5) {$v_3^0$};
\node at (6.2,1.5) {$v_3^+$};
\node at (6.9,1.5) {$v_4^-$};
\node at (7.5,1.5) {$v_4^0$};
\node at (8.2,1.5) {$v_4^+$};

\node at (2,-1.5) {$\bar v_1$};
\node at (5,-1.5) {$\bar v_2$};
\node at (8,-1.5) {$\bar v_3$};

\draw (base)--(c1); 
\draw (base)--(c2); 
\draw (base)--(c3); 

\draw (base)--(1m)--(1)--(1p)--(2m)--(2)--(2p)--(3m)--(3)--(3p)--(4m)--(4)--(4p);
\draw (c1)--(1m);
\draw (c1)--(2p);
\draw (c2)--(1m);
\draw (c2)--(2m);
\draw (c3)--(1p);
\draw (c3)--(3m);
\draw (c3)--(4m);

\end{tikzpicture}
\caption{The richly balanced colored graph of $\mc G_{\mc I}$ associated with the instance $\mc I$ of SAT of the example above.}
\end{center}
 \end{figure}

\subsection{Zero forcing property}

In this subsection we focus on a property of the graph topology called the zero forcing property, which is a sufficient condition for PI-GAS. To determine if a damping graph $\mc G=(V,V_d,E)$ satisfies the zero forcing property, we use the \emph{zero forcing algorithm} (ZFA). Initiating with the set of damped nodes being black and the undamped nodes being white, the algorithm applies repeatedly the \emph{black forcing rule} of selecting a black node $b$ with exactly one white neighbour $w$ and change the color of $w$ to black. Alternatively, we say that $b$ \textit{forces} $w$. The algorithm terminates if the black forcing rule cannot be applied anymore, which occurs if every black node has either none or at least two white neighbours. We refer to the graph after application of the zero forcing algorithm as the \emph{derived graph}. The set of black nodes in the derived graph is denoted by $\mc D(\mc G)$.

\begin{definition}  The original set of black nodes is said to be a \emph{zero forcing set} if $\mc D(\mc G)=V$, i.e. the derived graph consists entirely of black nodes. In that case, $\mc G$ is said to satisfy the zero forcing property. 
\end{definition}

We refer to \cite{zfs1}, \cite{zfs2}, \cite{zfs3} and \cite{zfs4} for more on zero forcing sets. In the following lemma we show that the set of uncolorable nodes is unchanged when all black nodes are turned into damped nodes. Then a small step is needed to show that $\mc G$ is PI-GAS if it satisfies the zero forcing property (\cref{eqle9}). 

\begin{lemma}\label{t10} Consider a damping graph $\mc G=(V,V_d,E)$ and define \\$\hat{\mc G}:=(V, \mc D(\mc G), E)$.  Then $\mc U(\hat{\mc G})=\mc U(\mc G).$
\end{lemma}

\begin{proof} We use an inductive argument to show that the black nodes in the derived graph of the ZFA are uncolorable, i.e. $\mc D(\mc G) \sbq \mc U(\mc G)$. Denote by $ \mc D^i(\mc G)$ the set of black nodes at the end of stage $i$ of the ZFA.  The black nodes in the initial graph of the ZFA are the damped nodes and they are black in every balanced colored graph, hence $\mc D^0(\mc G)=V_d \sbq \mc U(\mc G)$. Suppose that $i$ is not the last stage of the algorithm and $\mc D^i(\mc G) \sbq \mc U(\mc G)$. At the beginning of stage $i$, the algorithm selects and blackens a white node $w$ that is the only white neighbor of some black node $b$. Since $b$ and all its neighbors except for $w$ are uncolorable, $w$ must be black too in every balanced colored graph. So $\mc D^{i+1}(\mc G) \sbq \mc U(\mc G)$. By induction now follows that  $\mc D(\mc G) \sbq \mc U(\mc G)$. Now, consider any balanced colored graph of $\mc G$. Note that its black nodes cover $\mc U(\mc G)$ and hence $\mc D(\mc G)$ too. Since furthermore, the graph structures of $\mc G$ and $\hat{\mc G}$ are equal, the same color assignment can be used in $\hat{\mc G}$ to induce a balanced colored graph of $\hat{\mc G}$. Thus, any colorable node in $\mc G$ is colorable in $\hat{\mc G}$, i.e. $\mc U(\hat{\mc G}) \sbq \mc U(\mc G)$. The other direction follows readily from the fact that $V_d \sbq \mc D(\mc G)$. 
\qed\end{proof}

\begin{corollary}\label{eqle9} If $\mc G$ satisfies the zero forcing property, then $\mc G$ is PI-GAS. 
\end{corollary}

\begin{proof} If $\mc G=(V,V_d,E)$ satisfies the zero forcing property, then $\mc D(\mc G)=V$ and thus $\hat{\mc G}=(V,V,E)$. From \cref{t10} it follows that $\mc U(\mc G) = \mc U(\mc{\hat G})=V$. 
\qed\end{proof}

The converse of \cref{eqle9} is not true. A simple counterexample is a chordless 6-cycle, consisting of 3 damped and 3 undamped nodes, which are arranged alternately. Obviously, it does not satisfy the zero forcing property and still it is PI-GAS. Indeed, regardless of the way we would color the undamped nodes black, blue or red, no  balanced colored graph contains a blue or red node. Interestingly, a similar 8-cycle allows oscillatory behavior. 

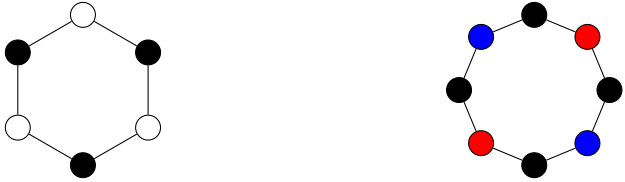
\begin{figure}[h!]
\begin{center}
\begin{tikzpicture}
[
b/.style={circle, draw=black,fill=black},
w/.style={circle, draw=black, fill=white},
bl/.style={circle, draw=black, fill=blue},
r/.style={circle, draw=black, fill=red}
]

\node[b] (1) at (0,0) {};
\node[bl] (2) at (.707,.293) {};
\node[b] (3) at (1,1) {};
\node[r] (4) at (.707,1.707) {};
\node[b] (5) at (0,2) {};
\node[bl] (6) at (-.707,1.707) {};
\node[b] (7) at (-1,1) {};
\node [r] (8) at (-.707,.293) {};

\draw (1)--(2)--(3)--(4)--(5)--(6)--(7)--(8)--(1);

\node[b] (9) at (-6,0) {};
\node[w] (10) at (-5.134,.5) {};
\node[b] (11) at (-5.134,1.5) {};
\node[w] (12) at (-6,2) {};
\node[b] (13) at (-6.866,1.5) {};
\node[w] (14) at (-6.866,.5) {};

\draw (9)--(10)--(11)--(12)--(13)--(14)--(9);

\end{tikzpicture}
\caption{The 6-cycle of alternating damped (black) and undamped (white) nodes is PI-GAS. On the other hand, the similar 8-cycle allows oscillatory behavior, as can be shown with the above richly balanced colored graph.}
\end{center}
\end{figure}

For the special case of tree graphs, the zero forcing property is also a necessary condition for $\mc G$ to be PI-GAS. Before we prove this, we need the following lemma:

\begin{lemma}\label{twoclasses} Consider damping graphs 
$\mc G_1=(V,V_{d},E_1)$ and  \\ $\mc G_2=(V,\mc U(\mc G_1),E_2)$, 
where $E_2$ is obtained from $E_1$ by deleting all or some edges in $\mc U(\mc G_1) \times \mc U(\mc G_1)$. Then $\mc U(\mc G_1) = \mc U(\mc G_2)$. 
\end{lemma}

\begin{proof} 
Since $\mc U(\mc G_1)$ is the set of damped nodes in $\mc G_2$, which are black in every balanced colored graph of $\mc G_2$, we have $\mc U(\mc G_1) \sbq \mc U(\mc G_2)$.  Now, consider a colorable node $i$ of $\mc G_1$ and a balanced colored graph of $\mc G_1$ in which $i$ is not black. Note that the uncolorable nodes of $\mc G_1$ are a subset of the black nodes in this colored graph and removing edges between black nodes preserves the balance of black nodes. Hence, by removing edges in $\mc U(\mc G_1) \times \mc U(\mc G_1)$ in a balanced colored graph of $\mc G_1$, the induced colored graph is still balanced. Therefore, $i$ is also colorable in $\mc G_2$. This shows that $\mc U(\mc G_2) \sbq \mc U(\mc G_1)$. 
\qed\end{proof}

\begin{proposition}\label{treeprop} Suppose that $\mc G$ is a tree graph. Then $\mc G$ satisfies the zero forcing property if and only if $\mc G$ is PI-GAS. 
\end{proposition}

\begin{proof} Let $\mc G=(V, V_d, E)$ be a damping graph of a tree. We show that $\mc D(\mc G)=\mc U(\mc G)$. Consider $\hat{\mc G}=(V,\mc D(\mc G),E)$. Denote by $\doublehat{\mc G}$ the damping graph that is obtained by removing all edges between damped nodes in $\hat{\mc G}$. By \cref{t10} and \ref{twoclasses}, we have $\mc U(\mc G)=\mc U(\hat{\mc G})=\mc U(\doublehat{\mc G})$. In $\hat{\mc G}$, each damped node has either none or at least two undamped neigbors, so the ZFA terminates immediately. By removing the links between damped nodes in $\hat{\mc G}$, this property is preserved, so $\mc D(\doublehat{\mc G})=\mc D(\hat{\mc G})=\mc D(\mc G)$. It remains to show that $\mc D(\doublehat{\mc G})= \mc U(\doublehat{\mc G})$. Any connected component of $\doublehat{\mc G}$ is a tree whose leaves are undamped nodes. Consider arbitrarily two leaves $l_1,l_2$ on such a connected component. On the path  $l_1 \to l_2$, there are no two consecutive damped nodes. Color the undamped nodes alternately blue and red. Then for any other leaf $l_i$ of the corresponding connected component in $\doublehat{\mc G}$, consider $l_1 \to l_i$ and observe that only the last part of this path has not been colored. Hence we can color the undamped nodes on this last part in such a way that the undamped nodes of the path $l_1 \to l_i$ are colored alternately blue and red. Repeat this until all undamped nodes on paths from $l_1$ to leaves in the same connected component have been colored. As $\doublehat{\mc G}$ is a tree, these paths cover all nodes. By also coloring the damped nodes black, we obtain a balanced colored graph in which all damped nodes are richly balanced and the undamped nodes are red or blue and therefore colorable. Thus, the white nodes in the derived graph of $\doublehat{\mc G}$, which are the undamped nodes in this graph, are colorable. Consequently, $ V \backslash \mc D(\doublehat{\mc G}) \sbq V \backslash \mc U(\doublehat{\mc G})$. From the proof of \cref{t10}, $\mc D(\doublehat{\mc G}) \sbq \mc U(\doublehat{\mc G})$ and hence $\mc D(\doublehat{\mc G})= \mc U(\doublehat{\mc G})$. We conclude that $\mc D(\mc G)=\mc U(\mc G)$, so $\mc D(\mc G)=V$  if and only if $\mc U(\mc G)=V$. 
\qed\end{proof}

\su{Chord node coloring}

The fact that the RBC problem is NP-complete, does not imply that the existence of a richly balanced colored graph can only be determined by a brute-force approach that involves all undamped nodes. For (large) graphs with a low number of fundamental cycles, we can significantly reduce the search space. In this subsection, we present the \text{chord node coloring (CNC) algorithm} that solves the RBC problem in such a way that the number of variables used in the brute-force search is proportional to the fundamental cycles in the graph. The CNC algorithm only runs through combinations of colors of the \textit{chord nodes}: undamped nodes that are the endpoints of the chords of the graph. It can be verified in quadratic time if for a given coloring of chord nodes there exists a richly balanced colored graph. Notice that the number of chord nodes is not more than twice the number of fundamental cycles. 

Given the damping graph $\mc G$, the algorithm starts with a given coloring of the chord nodes of the reduced graph $\doublehat{\mc G}$, as defined in the proof of \cref{treeprop}, in the original colors red, blue and black. All damped nodes are black, while other nodes are white to indicate that its coloring in one of the original colors is yet to be determined. White nodes are recolored according to the forcing rules:

\smallskip

\bi
\item \textit{black forcing rule:} If there is a black node $b$ that has precisely one white neighbor $w$ and no red or blue neighbors, color $w$ black. $b$ is called a \textit{(black-) forcing node}. 
\item \textit{color forcing rule:} If there is a black node $b$ that has precisely one white neighbor $w$, at least one red (blue) neighbor  and no blue (red) neighbors, color $w$ blue (red). $b$ is called a \textit{(color-)forcing node}. 
\ei

\smallskip

The forcing rules are necessary conditions for creating a balanced colored graph, in the sense that coloring the white node otherwise would immediately result in an unbalanced node, which, in this context, is a node that does not have white neighbors and either red or blue neighbors but not both. In the sequel, the graph that is obtained after repeatedly applying the forcing rules (in arbitrary order) is referred to as the derived graph. Notice that the number of times the forcing rules need to be applied to obtain the derived graph is at most $n$, where in each iteration, finding a forcing node and coloring the white node can be done in linear time. Thus, the time needed to find the derived graph for a single chord node coloring can be done in polynomial (quadratic) time. In \cref{correct}, we show that the original graph is PI-GAS if and only if there exists a chord node coloring for which the derived graph of one of the connected components of $\doublehat{\mc G}$ does not contain unbalanced black nodes and is not completely black. The chord node coloring algorithm is summarized as follows: 

\newpage

\begin{algorithm}
\caption{Chord node coloring (CNC) algorithm}
\begin{algorithmic}
\STATE{Input: the damping graph $\mc G=(V,V_d,E)$}
\STATE{Create the possibly disconnected colored graph $\doublehat{\mc G}=(V,\mc D(\mc G),\doublehat E)$, where $\doublehat E$ is obtained from $E$ by deleting all edges in $\mc D(\mc G) \times \mc D(\mc G)$}
\STATE{Color the nodes in $\mc D(\mc G)$ black}
\FOR{each connected component $\doublehat{\mc G}_i=(V_i, V_i \cap \mc D(\mc G), \doublehat E \cap (V_i \times V_i))$ of $\doublehat{\mc G}$}
\STATE{Find a spanning tree of $\doublehat{\mc G}_i$}
\STATE{Let $V_i^c$ be the set of undamped nodes in $V_i \backslash \mc D(\mc G)$ that are the endpoints of the chords of the spanning tree}
\FOR{each color combination in $\{$black, blue, red$\}^{|V_i^c|}$}
\STATE{Color the nodes in $V_i^c$ according to this combination of colors}
\STATE{Color the nodes in $V_i \backslash \haak{\mc D(\mc G) \cup V_i^c}$ white}
\STATE{Apply repeatedly the black and color forcing rules to $\doublehat{\mc G}_i$, until no forcing nodes are left}
\STATE{If the derived graph of $\doublehat{\mc G}_i$ does not contain unbalanced black nodes and not all nodes are black, terminate and return: ``$\mc G$ is not PI-GAS"}
\ENDFOR
\ENDFOR
\STATE{Return: ``$\mc G$ is PI-GAS"}
\end{algorithmic}
\end{algorithm}

Let us verify the algorithm:

\begin{theorem}\label{correct} The RBC problem is solved correctly by the CNC algorithm. 
\end{theorem}

\begin{proof}  For this proof, we categorize black nodes as follows\footnote{Due to the symmetry regarding blue and red nodes, the defining conditions for a node to be of the type described in the tabular, must be extended with the same conditions where the results for red and blue nodes are interchanged. Also, note that this tabular is in accordance with the balancing definitions for colored graphs.}:

\begin{tabular}{l| lllll}
number of neighbors of the color $\rightarrow$  &&&& \\
black node type $\downarrow$  & black & white & blue & red \\\hline
poorly balanced & 0+ & 0 & 0 & 0 \\
richly balanced & 0+ & 0+ & 1+ & 1+ \\
unbalanced & 0+ & 0 & 1+ & 0 \\
poorly indefinite & 0+ & 2+ & 0 & 0 \\
richly indefinite & 0+ & 2+ & 1+ & 0 \\
black-forcing & 0+ & 1 & 0 & 0 \\
color-forcing & 0+ & 1 & 1+ & 0 \\
\end{tabular}

Suppose first that the algorithm claims that $\mc G$ is PI-GAS. Then for any connected component $\doublehat{\mc G}_i$ and assignment of colors to $V^c_i$, applying the two forcing rules lead to unbalanced black nodes or all nodes being black. In both cases, there does not exist a richly balanced colored graph with the given chord node coloring. Since this holds for all components and color assignments to $V^c_i$, the only balanced colored graph of $\doublehat{\mc G}$ consist entirely of black nodes, meaning that $\mc U(\doublehat{\mc G})=V$. Recall from the proof of \cref{treeprop} that $\mc U(\doublehat{\mc G}) =\mc U(\mc G)$. Hence, $\mc U(\mc G)=V$, i.e. there does not exist a richly balanced colored graph of $\mc G$. 

Now suppose that the algorithm claims that $\mc G$ is not PI-GAS. Consider the derived graph of the connected component $\doublehat{\mc G}_i$, which does not contain unbalanced black nodes and is not completely black. 
We remove the following elements of $\doublehat{\mc G}_i$:
\smallskip
\begin{enumerate}
\item All blue and red nodes and their incident edges
\item All edges between black nodes
\item All edges between richly balanced black nodes and white nodes
\end{enumerate}
\smallskip
Doing this creates a forest, referred to as $\breve{\mc G}_i$. Indeed, each connected component of $\breve{\mc G}_i$ must be a tree since all chords are removed and thus, no fundamental cycles are left. Note furthermore that every balanced black node in $\doublehat{\mc G}_i$ is an isolated node in $\breve{\mc G}_i$, whereas every indefinite black node in $\doublehat{\mc G}_i$ is poorly indefinite in $\breve{\mc G}_i$. Since furthermore, $\doublehat{\mc G}_i$ does not contain forcing nodes, every black node in $\breve{\mc G}_i$ contains none or at least two white neighbors. Consider one of the connected components $(\breve{\mc G}_i)_j$. 
Its white nodes can be colored according to the procedure in the proof of \cref{treeprop}. By reconnecting all removed edges and nodes, all black nodes in $\doublehat{\mc G}_i$ are balanced: previously richly balanced black nodes in $\doublehat{\mc G}_i$ remain richly balanced and the same holds for poorly balanced nodes. However, previously indefinite black nodes in $\doublehat{\mc G}_i$ become richly balanced, since they were rendered richly balanced in $\breve{\mc G}_i$ by coloring at least one white neighbor blue and another one red. Thus, the colored graph of $\doublehat{\mc G}_i$ is richly balanced and we can render $\doublehat{\mc G}$ richly balanced by coloring the nodes of other connected components $\doublehat{\mc G}_j$ black.  Since the nodes in $\mc D(\mc G)$ are black and the graph structures of $\mc G$ and $\doublehat{\mc G}$ are equal except for the edges in $\mc D(\mc G) \times \mc D(\mc G)$, the same assignment of colors can be used to the nodes of $\mc G$ to obtain a richly balanced graph of $\mc G$.  
\qed\end{proof}

{\bf Remark } The number of chord nodes depends on the chosen spanning tree. In order to reduce the completion time of the algorithm, we want to find a selection of chord nodes with a minimum number of undamped nodes, e.g. by avoiding chords that have two undamped endpoints and selecting as many chords as possible that share an undamped node.

 \V

{\bf Example } An example of a damping graph that is not PI-GAS is given in the figure below. 

\begin{figure}[h!]
\begin{center}
\begin{tikzpicture}
[
b/.style={circle, draw=black,fill=black},
w/.style={circle, draw=black, fill=white},
bl/.style={circle, draw=black, fill=blue},
r/.style={circle, draw=black, fill=red},
vw/.style={rectangle, draw=black, fill=white},
vwb/.style={rectangle, draw=black, fill=black},
vwr/.style={rectangle, draw=black, fill=red},
]

\node [w] (1) at (1.991730,4.397093) {};
\node [w] (2) at (3.058293,4.890728) {};
\node [b] (3) at (4.094515,4.414921) {};
\node [w] (4) at (4.871306,3.222695) {};
\node [b] (5) at (4.890551,2.069705) {};
\node [w] (6) at (4.104763,1.259848) {};
\node [w] (7) at (3.127251,0.650625) {};
\node [w] (8) at (1.626551,0.548915) {};
\node [b] (9) at (1.078613,3.637908) {};
\node [w] (10) at (1.403538,2.718135) {};
\node [w] (11) at (1.442377,1.727159) {};
\node [b] (12) at (2.326330,1.275892) {};
\node [w] (13) at (2.308036,2.243963) {};
\node [w] (14) at (3.229290,1.885292) {};
\node [b] (15) at (4.014391,2.486022) {};
\node [w] (16) at (3.882363,3.466819) {};
\node [b] (17) at (2.989042,3.888577) {};
\node [w] (18) at (2.178578,3.327839) {};
\node [b] (19) at (3.063728,2.870270) {};
\node [b] (20) at (0.639040,1.127930) {};
\node [w] (21) at (0.479714,2.289665) {};

\node [vwr] (1a) at (7.991730,4.397093) {};
\node [r] (2a) at (9.058293,4.890728) {};
\node [b] (3a) at (10.094515,4.414921) {};
\node [bl] (4a) at (10.871306,3.222695) {};
\node [b] (5a) at (10.890551,2.069705) {};
\node [r] (6a) at (10.104763,1.259848) {};
\node [bl] (7a) at (9.127251,0.650625) {};
\node [r] (8a) at (7.626551,0.548915) {};
\node [b] (9a) at (7.078613,3.637908) {};
\node [b] (10a) at (7.403538,2.718135) {};
\node [b] (11a) at (7.442377,1.727159) {};
\node [b] (12a) at (8.326330,1.275892) {};
\node [b] (13a) at (8.308036,2.243963) {};
\node [vwb] (14a) at (9.229290,1.885292) {};
\node [b] (15a) at (10.014391,2.486022) {};
\node [vwb] (16a) at (9.882363,3.466819) {};
\node [b] (17a) at (8.989042,3.888577) {};
\node [b] (18a) at (8.178578,3.327839) {};
\node [b] (19a) at (9.063728,2.870270) {};
\node [b] (20a) at (6.639040,1.127930) {};
\node [bl] (21a) at (6.479714,2.289665) {};

\draw (1)--(2)--(3)--(4)--(5)--(6)--(7)--(8)--(20);
\draw (20)--(21);
\draw (21)--(9);
\draw (9)--(1);
\draw (9)--(10)--(11)--(12)--(13)--(14)--(15)--(16);
\draw (16)--(17);
\draw (17)--(18)--(19);
\draw (19)--(14);
\draw (10)--(18);
\draw (3)--(16);
\draw (5)--(15);

\draw (1a)--(2a);
\draw[dashed] (2a)--(3a)--(4a)--(5a)--(6a)--(7a);
\draw (7a)--(8a);
\draw [thick,red,<-] (8a)--(20a);
\draw (20a)--(21a);
\draw [blue,thick,<-] (21a)--(9a);
\draw[dotted,thick] (9a)--(1a);
\draw (9a)--(10a);
\draw[thick,->] (10a)--(11a);
\draw (11a)--(12a)--(13a);
\draw[thick,<-] (13a)--(14a);
\draw (14a)--(15a)--(16a);
\draw[dotted,thick]  (16a)--(17a);
\draw[thick,->] (17a)--(18a);
\draw (18a)--(19a);
\draw[dotted,thick] (19a)--(14a);
\draw[thick,<-] (10a)--(18a);
\draw[dotted,thick] (3a)--(16a);

\end{tikzpicture}
\caption{A damping graph that is shown to admit a richly balanced coloring. By \cref{t3}, it is not PI-GAS.}
\end{center}
\end{figure}
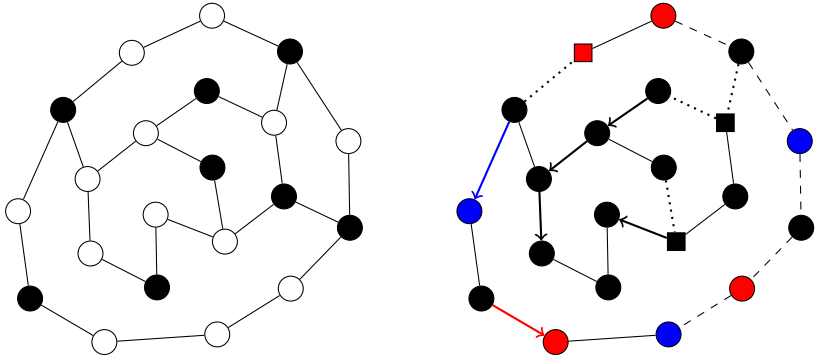

The figure on the left shows a damping graph $\mc G$ where black and white nodes represent damped and undamped nodes, respectively. The chord node algorithm first applies the ZFA, which terminates immediately and then it deletes the single edge between two black nodes. The next step is selecting the chords, which are depicted by the dotted edges in the figure on the right. The undamped chord nodes are depicted as squares. For the illustrated combination of red, black and black chord nodes, the two forcing rules are applied (depicted by arrows). This does not lead to unbalanced nodes. The nodes that are still white at this step are connected to each other through the dashed edges. The procedure described in the proof of \cref{treeprop} describes how to color them red or blue. The result is the richly balanced colored graph on the right.

Let us design the system parameter values in such a way that the velocities of the oscillating nodes are equal. Let $r=1$ and consider the vector $v \in \R^n$ such that $v_i =-1,0,1$ for blue, black and red nodes, respectively. The system parameter values can now be obtained from the proof of \cref{o}. This leads to an invariant set $\mc S^{LS}$ of dimension 2, which accounts for one oscillating group having one dimension of freedom for momenta and one for spring elongations. Therefore, any initial condition will lead to a partial synchronization of the groups of red, blue and black nodes (see Figure 4). The oscillating group has an angular frequency of 1, as can be verified by analyzing the eigenvalues of the matrix $\breve A$ (see the remark under \cref{nwg}). A simulation of the same network with perturbed parameter values that render the system GAS is shown in Figure 5.\footnote{With probability 1, for uniformly distributed randomly chosen parameter values, the matrix $M^{-1}\mathds L$ in \eqref{iii} has distinct eigenvalues with 1-dimensional eigenspaces that are not contained in the subspace $\bp 0 \; I_{n_ur} \ep^T$. Therefore, perturbations of parameter values typically render the system GAS.}

\begin{figure}[h!]
\includegraphics[width=12cm]{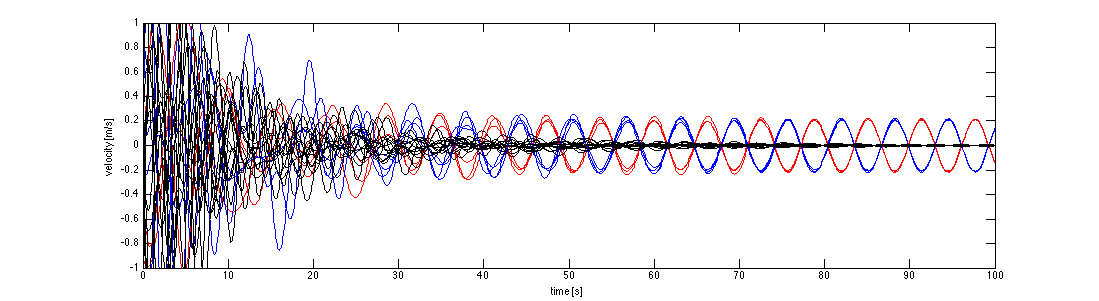}
\caption{Velocity of the nodes of the graph in Figure 3 with system parameters designed to oscillate with equal velocities. Black, red and blue lines correspond to black, red and blue nodes. The red and blue nodes form one oscillating group.}
\end{figure}

\vspace{-5mm}
\begin{figure}[h!]
\includegraphics[width=12cm]{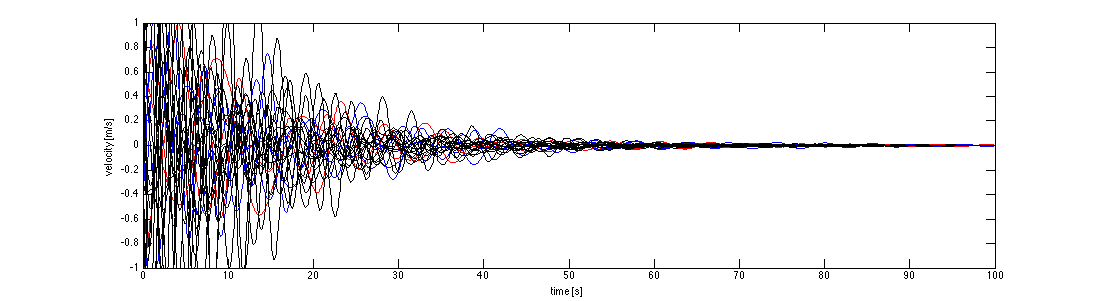}
\caption{Velocities of the nodes of the same network, with perturbed system parameters that render the system GAS.}
\end{figure}

\vspace{-5mm}

By placing dampers at strategic locations in the graph, the system becomes PI-GAS. This occurs for example if the uppermost undamped node is changed into a damped node, see Figure 6. This node forces his white neighbor to black (depicted by the arrow). After removal of the edges between black nodes, we find one major connected component with white nodes. By selecting the chords of Figure 5 within the set of remaining edges (dotted edges), two undamped chord nodes appear (yellow squares). If one of them is colored black, the black forcing rule colors the whole component black. Therefore, the only remaining option is to color the chord nodes red and blue. The color forcing rule will color the node depicted by the yellow circle red or blue, but in either case, an unbalanced black node will appear. By the CNC algorithm, the system is PI-GAS. 

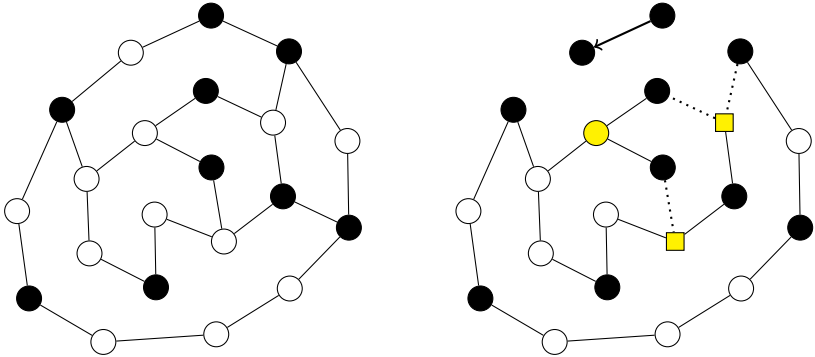
\begin{figure}[h!]
\begin{center}
\begin{tikzpicture}
[
b/.style={circle, draw=black,fill=black},
w/.style={circle, draw=black, fill=white},
bl/.style={circle, draw=black, fill=blue},
r/.style={circle, draw=black, fill=red},
yy/.style={circle, draw=black, fill=yellow},
vw/.style={rectangle, draw=black, fill=white},
vwb/.style={rectangle, draw=black, fill=black},
vwr/.style={rectangle, draw=black, fill=red},
vwy/.style={rectangle, draw=black, fill=yellow},
]

\node [w] (1) at (1.991730,4.397093) {};
\node [b] (2) at (3.058293,4.890728) {};
\node [b] (3) at (4.094515,4.414921) {};
\node [w] (4) at (4.871306,3.222695) {};
\node [b] (5) at (4.890551,2.069705) {};
\node [w] (6) at (4.104763,1.259848) {};
\node [w] (7) at (3.127251,0.650625) {};
\node [w] (8) at (1.626551,0.548915) {};
\node [b] (9) at (1.078613,3.637908) {};
\node [w] (10) at (1.403538,2.718135) {};
\node [w] (11) at (1.442377,1.727159) {};
\node [b] (12) at (2.326330,1.275892) {};
\node [w] (13) at (2.308036,2.243963) {};
\node [w] (14) at (3.229290,1.885292) {};
\node [b] (15) at (4.014391,2.486022) {};
\node [w] (16) at (3.882363,3.466819) {};
\node [b] (17) at (2.989042,3.888577) {};
\node [w] (18) at (2.178578,3.327839) {};
\node [b] (19) at (3.063728,2.870270) {};
\node [b] (20) at (0.639040,1.127930) {};
\node [w] (21) at (0.479714,2.289665) {};

\node [b] (1a) at (7.991730,4.397093) {};
\node [b] (2a) at (9.058293,4.890728) {};
\node [b] (3a) at (10.094515,4.414921) {};
\node [w] (4a) at (10.871306,3.222695) {};
\node [b] (5a) at (10.890551,2.069705) {};
\node [w] (6a) at (10.104763,1.259848) {};
\node [w] (7a) at (9.127251,0.650625) {};
\node [w] (8a) at (7.626551,0.548915) {};
\node [b] (9a) at (7.078613,3.637908) {};
\node [w] (10a) at (7.403538,2.718135) {};
\node [w] (11a) at (7.442377,1.727159) {};
\node [b] (12a) at (8.326330,1.275892) {};
\node [w] (13a) at (8.308036,2.243963) {};
\node [vwy] (14a) at (9.229290,1.885292) {};
\node [b] (15a) at (10.014391,2.486022) {};
\node [vwy] (16a) at (9.882363,3.466819) {};
\node [b] (17a) at (8.989042,3.888577) {};
\node [yy] (18a) at (8.178578,3.327839) {};
\node [b] (19a) at (9.063728,2.870270) {};
\node [b] (20a) at (6.639040,1.127930) {};
\node [w] (21a) at (6.479714,2.289665) {};

\draw (1)--(2)--(3)--(4)--(5)--(6)--(7)--(8)--(20);
\draw (20)--(21);
\draw (21)--(9);
\draw (9)--(1);
\draw (9)--(10)--(11)--(12)--(13)--(14)--(15)--(16);
\draw (16)--(17);
\draw (17)--(18)--(19);
\draw (19)--(14);
\draw (10)--(18);
\draw (3)--(16);
\draw (5)--(15);

\draw[thick, <-] (1a)--(2a);
\draw (3a)--(4a)--(5a)--(6a)--(7a);
\draw (7a)--(8a);
\draw  (8a)--(20a);
\draw (20a)--(21a);
\draw (21a)--(9a);
\draw (9a)--(10a);
\draw (10a)--(11a);
\draw (11a)--(12a)--(13a);
\draw (13a)--(14a);
\draw (14a)--(15a)--(16a);
\draw[dotted,thick] (16a)--(17a);
\draw (17a)--(18a);
\draw (18a)--(19a);
\draw[dotted,thick] (19a)--(14a);
\draw (10a)--(18a);
\draw[dotted,thick] (3a)--(16a);

\end{tikzpicture}
\caption{The network of Figure 5, where the uppermost undamped node is changed into a damped node. The CNC algorithm certifies that this network is PI-GAS.}
\end{center}
\end{figure}

\section{Conclusion}

Asymptotic output consensus of a mass-spring-damper system with damped and undamped nodes amounts to global asymptotic stability (GAS) of a linearly shifted system.  For a given set of system parameters, the consensus problem is equivalent to an eigenspace problem that  depends on the graph topology and the edge weights, as well as on the mass and resistance matrices of the undamped nodes (\cref{t1}). This result is taken as starting point for the parameter-independent output consensus problem, which asks if all systems with the same underlying graph and set of damped nodes are GAS. We showed that this problem is equivalent to an NP-complete graph coloring problem called the richly balanced coloring problem. The zero forcing property is a sufficient condition for the system to be PI-GAS and thus, to guarantee output consensus for all system parameter values. For tree graphs, it is also a necessary condition. The search space of the brute-force approach to decide if the system is PI-GAS can be confined to the undamped nodes that are endpoints of a set of edges that form the chords of the graph. This allows one to decide fast whether a large network with a few fundamental cycles guarantees output consensus. The results obtained in this paper open the way for an analysis of qualitatively heterogeneous networks. It can be expected to influence also structural controllability analysis and pinning control, focusing on the problem of where to allocate a minimum amount of dampers in order to render the system PI-GAS.

\nocite{*}
\bibliographystyle{siamplain}
\bibliography{references}

\end{document}